\documentclass[pdflatex,sn-mathphys,Numbered]{sn-jnl}

\usepackage{graphicx}
\usepackage{multirow}
\usepackage{amsmath,amssymb,amsfonts}
\usepackage{amsthm}
\usepackage{mathrsfs}
\usepackage[title]{appendix}
\usepackage{textcomp}
\usepackage{manyfoot}

\usepackage{etoolbox}   
\makeatletter
\patchcmd{\@maketitle}{\artauthors}{{\artauthors}}{}{}
\makeatother

\theoremstyle{plain}\newtheorem{theorem}{Theorem}
\theoremstyle{plain}\newtheorem{proposition}[theorem]{Proposition}
\theoremstyle{plain}\newtheorem{lemma}[theorem]{Lemma}
\theoremstyle{plain}\newtheorem{corollary}[theorem]{Corollary}
\theoremstyle{plain}

\theoremstyle{definition}\newtheorem{definition}[theorem]{Definition}
\theoremstyle{remark}\newtheorem{remark}[theorem]{Remark}
\theoremstyle{definition}
\theoremstyle{definition}\newtheorem{assumption}{Assumption}

\numberwithin{theorem}{section} 
\numberwithin{equation}{section}

\usepackage{mathtools}
\mathtoolsset{showonlyrefs}
\DeclarePairedDelimiter{\norm}{\lVert}{\rVert}

\renewcommand{\d}{\mathrm{d}}
\usepackage[dvipsnames]{xcolor}

\begin{document}

\title[
	Fixed energy solutions to
	the Euler-Lagrange equations 
	of an indefinite Lagrangian with 
	affine Noether charge
]{
	Fixed energy solutions to
	the Euler-Lagrange equations 
	of an indefinite Lagrangian with \\
	affine Noether charge
}

\author[1]{\fnm{Erasmo} \sur{Caponio}}\email{erasmo.caponio@poliba.it}

\author[2]{\fnm{Dario} \sur{Corona}}\email{dario.corona@unicam.it}

\author[2,3]{\fnm{Roberto} \sur{Giamb\`o}}\email{roberto.giambo@unicam.it}

\author[4]{\fnm{Paolo} \sur{Piccione}}\email{paolo.piccione@usp.br}

\affil[1]{
	\orgdiv{Dipartimento di Meccanica, Matematica e Management},
	\orgname{Politecnico di Bari},
	\orgaddress{\city{Bari}, \country{Italy}}}
\affil[2]{
	\orgdiv{School of Science and Technology},
	\orgname{University of Camerino},
	\orgaddress{\city{Camerino}, \country{Italy}}}
\affil[3]{
	\orgdiv{Istituto Nazionale di Fisica Nucleare},
	\orgname{Sezione di Perugia},
	\orgaddress{\city{Perugia}, \country{Italy}}}
\affil[4]{
	\orgdiv{Institute of Mathematics and Statistics},
	\orgname{Universidade de S\~ao Paulo}, 
\orgaddress{\city{S\~ao Paulo}, \country{Brazil}}}

\abstract{
	We consider an autonomous,
	indefinite  Lagrangian $L$ admitting an infinitesimal
	symmetry $K$ whose associated Noether charge
	is linear in each tangent space.
	Our focus lies in investigating solutions to the  
	Euler-Lagrange equations having fixed energy and
	that connect a given point $p$ to a flow line $\gamma=\gamma(t)$
	of $K$ that does not cross $p$.
	By utilizing the invariance of $L$ under the flow of $K$,
	we simplify  the problem  into a two-point boundary problem.
	Consequently, we derive an equation that involves
	the differential of the ``arrival time'' $t$,
	seen as a functional  on the infinite dimensional manifold
	of connecting paths satisfying the semi-holonomic
	constraint defined by the Noether charge.
	When $L$ is  positively homogeneous of degree $2$ in the velocities,
	the resulting equation establishes 
	a variational principle that extends the Fermat's principle
	in a stationary spacetime.
	Furthermore, we also analyze the scenario
	where the Noether charge is affine.
}

\maketitle
\bmhead{Acknowledgments}
E. Caponio is partially supported by PRIN PNRR P2022YFAJH 		
{\em ``Linear and Nonlinear PDE's: New directions and Applications''}.

\noindent D. Corona is partially supported by the FAPESP (S\~ao Paulo, Brazil) grant n. 2022/13010-3.

\noindent E. Caponio, D. Corona, R. Giamb\`o thank the partial support of GNAMPA INdAM - Italian National Institute of High Mathematics,
project CUP-E55F22000270001. 

\noindent P. Piccione is partially sponsored by FAPESP (S\~ao Paulo, Brazil),
grant n. 2022/16097-2, and by CNPq, Brazil.

\noindent We would like to thank a referee for providing us with  valuable comments.

\section{Introduction}
In the problem of finding  solutions with fixed energy for an autonomous
Lagrangian system with a finite number of degrees of freedom,
subject to two-point or periodic boundary conditions,
one viable approach is to allow for a
free interval of parametrization for the involved curves.
This method entails employing the action functional with fixed energy,
as originally defined by Ma\~n\'e
(see~\cite{mane97} and the survey \cite{Abbondandolo2013}).
Specifically,
given a pair of distinct points on a compact manifold $M$
and a fiberwise convex and superlinear Lagrangian $L$,
there always exists a solution connecting the two points
with a fixed energy value $\kappa$,
provided that $\kappa$ is strictly greater than the
so-called {\em Ma\~n\'e critical value}
$c(L)$~\cite[Theorem X]{CoDeIt97}.
This approach has also demonstrated its effectiveness in addressing
the challenging problem of establishing the existence of periodic solutions,
as explored in works such as
\cite{Contreras2006Calcvar,Corona2020JFPT,Corona2020Symmetry,AsBeMa21}.

A more basic approach involves fixing the parameter interval and the initial point $p$,
while allowing the final point to traverse along a given curve $\gamma$.
This is a common framework employed in General Relativity when studying causal geodesics,
which represent the paths of light rays (photons) or
the worldlines of massive particles.
In this scenario, the Lagrangian energy coincides with the conserved quantity of the geodesic,
namely the square of the norm of its velocity vector.
Consequently, the values $\kappa=0, -1$ correspond to
the energy levels of light rays and massive particles, respectively.
The curve $\gamma$ represents the worldline of an observer,
while $p$ signifies the event of emission or, in an alternative perspective,
$\gamma$ symbolizes the worldline of a source of light signals or massive particles,
and $p$ represents the event of detecting those signals
(in the latter case, the geodesics originate from $\gamma$ and terminate at $p$).
In the former scenario, the parameter value of $\gamma$
(more precisely,  its ``proper time'')
at the intersection point with a lightlike future-oriented curve
$z$ from $p$ to $\gamma$ is referred to as the \emph{arrival time} of $z$ \cite{kovner90}.
The future-oriented lightlike geodesics are then all and only
the stationary points of the arrival time with respect to any smooth
variation $z_\epsilon$ made by smooth future-oriented lightlike curves between $p$ and $\gamma$
\cite{Perlick19901319}.
This general statement is recognized as \emph{Fermat's principle in General Relativity}. 

In fact, when the spacetime is static or stationary, Levi-Civita first
introduced Fermat's principle in local coordinates in \cite{LeviCi17} and
\cite{LeviCi18} and observed that the geometry of lightlike geodesics in the
spacetime can be linked to  a metric in a spacelike slice, called {\em optical
metric}, which is  Riemannian, when the spacetime is static,  and Finslerian in
the stationary case, and that  allows describing and calculating various
geometric and causal properties of the spacetime through it 
\cite{Pham57,GibWer08,GiHeWW09,CaponioMasiello2011,Caponio2011}.
Subsequently, other variations of Fermat's principle have emerged, sharing fundamental elements while possessing distinct technical characteristics
(see, e.g.,
\cite{Quan56,Uhlenb75,fortunato1995,AntPic96,PerPic98,Frolov:2013sxa,CaJaSaIP});
moreover,  it has been generalized to massive particles
in~\cite{giannoni1998-calcvar,CaponioMasiello2011}.
Light rays and massive particles via variational methods have also been studied
in \cite{Giannoni1997375,giannoni1998-AnnInstPoin,Giannoni1999849,giannoni2000-JGeoPhy,Giannoni2002563,CaJaMa10}.
The use of Fermat's principle also emerges in the study of motion around black holes
\cite{Hod:2018elr},
as well as in the well-known phenomenon known as \emph{gravitational lensing}
\cite{Faraoni1992425,Nandor199645,Frittelli20021230071,Sereno:2003tk,Giambo2004117},
which has been specialized for particular exact solutions of the Einstein field equations,
such as Schwarzschild \cite{Virbhadra:1999nm} and NUT spacetimes \cite{Halla2020}.
Interestingly, the utilization of least-time principles in observational
cosmology opens up the possibility of interpreting observed instances of
gravitational lensing without the need to invoke the existence of dark matter \cite{Annila20112944}. 
We recommend referring readers to \cite{Perlick-lrr}
for a comprehensive review on gravitational lensing in a relativistic context. 

Additionally,
it is worth noting that the application of variational principles related with or inspired
by Fermat's principle has been extended and generalized beyond Euclidean or
Lorentzian geometry as evidenced
in~\cite{Perlick2006365,Duval08,Masiello2009783,GaPiVi12,CapSta16,caponio2018,Minguz19,Javaloyes2020,Herrera:2018qng,CaJaSaIP, caponio2021}.

This work contributes to the  latest type of research. We consider an
indefinite Lagrangian $L$ on a manifold $M$ that is invariant under a
one-dimensional group of local diffeomorphisms
generated by a complete vector field $K$.
The  Noether charge associated with $L$ is assumed to be linear in each tangent
space $T_xM$.
Our focus lies on solutions to the Euler-Lagrange equations of
the action functional of $L$ that connect a point $p$ to a flow line $\gamma$
of $K$ and having fixed energy $\kappa$.
Our approach is based on the variational setting in  \cite{caponio2023-calcvar},
which is inspired by \cite{giannoni1999}.
The main result in this work, Theorem~\ref{theorem:almostFermat}, at least when $L$ is $2$-homogeneous in the velocities,  extends   Fermat's principle as established in \cite{caponio2002-diffgeo},  that was 
specifically tailored to the framework explored in \cite{giannoni1999}.
Leveraging this extension, we provide proof of existence (Theorem~\ref{existmult}-(a))
and a multiplicity result (Theorem~\ref{existmult}-(b)) for such solutions.
Additionally, we delve into the analysis of the case where the
Noether charge is an affine function in Appendix~\ref{sec:affineNoether}.
In Appendix B we give some results
that link an assumption on the manifold of curves
that we consider in our variational setting,
called {\em pseudocoercivity},
with the notion of global hyperbolicity
for the cone structure associated with $L$ in the $2$-homogeneous case.

\section{Notations, assumptions and a class of examples}
\label{sec:notation}
Let $M$ be a smooth, connected manifold of dimension $(m+1)$,
where $m \geq 1$.
We denote the tangent bundle of $M$ as $TM$.
In this paper, we consider a Riemannian metric $g$ on $M$ as an auxiliary metric,
and we use $\norm{\cdot}\colon TM \to \mathbb{R}$ to represent its induced norm;
specifically, for any $v \in TM$, we have $\norm{v}^2 = g(v,v)$.
We represent an element of $TM$ as a pair $(x,v)$,
where $x$ belongs to $M$ and $v$ belongs to the tangent space $T_xM$.

Let $L\colon TM \to \mathbb{R}$ be a Lagrangian on $M$.
For any $(x,v) \in TM$,
we denote the vertical derivative of $L$ as $\partial_vL(x,v)[\cdot]$,
which is is defined as follows:
\begin{equation*}
	\label{eq:def-dvL}
	\partial_vL(x,v)[\xi] = \frac{d}{ds}L(x,v + s\xi)\bigg|_{s = 0},
	\quad \forall \xi \in T_xM.
\end{equation*}
We also need a derivative w.r.t. $x$, denoted by $\partial_xL(x,v)$.
This is defined only locally (in a system of coordinates) as:
\begin{equation*}
	\label{eq:def-dxL}
	\partial_xL(x,v)[\xi]
	= \sum_{i = 0}^{m} \frac{\partial L}{\partial x^i}(x,v)\xi^i,
	\quad \forall \xi \in T_xM,
\end{equation*}
where $(x^0,\dots,x^{m})$ is a local coordinate system in a neighbourhood of $x$,
and consequently, $(x^0,\dots,x^{m},v^0,\dots,v^m)$ are the induced coordinates on $TM$.
With this notation, the Euler-Lagrange equations for a curve
$z\colon [0,1] \to M$ of class $C^1$ are given by:
\begin{equation}
	\label{eq:Euler-Lagrange}
	\frac{d}{ds}\partial_vL(z,\dot{z}) - \partial_xL(z,\dot{z}) = 0,
	\quad \forall s \in [0,1],
\end{equation}
where $\dot{z}$ denotes the derivative of $z$ with respect to the parameter $s$.
It is well-known that the \emph{energy function} $E\colon TM \to \mathbb{R}$, defined as:
\begin{equation*}
	\label{eq:def-E}
	E(x,v) = \partial_vL(x,v)[v] - L(x,v),
\end{equation*}
is a first integral of the Lagrangian system.
Therefore, if $z\colon [0,1] \to M$ is a solution of the Euler-Lagrange equations,
there exists a constant $\kappa \in \mathbb{R}$ such that:
\begin{equation}
	\label{eq:E-const}
	E(z,\dot{z}) = \kappa,
	\quad \forall s \in [0,1].
\end{equation}

\begin{assumption}
	\label{ass:L}
	The Lagrangian $L\colon TM \to \mathbb{R}$
	satisfies the following conditions:
	\begin{itemize}
		\item $L$ is a $C^1$ function on $TM$;
		\item There exists a complete $C^3$ vector field $K$ on $M$ such that
			$L$ is invariant under the one-parameter group of
			$C^3$ diffeomorphisms of $M$ generated by $K$ 
			(we refer to $K$ as an \emph{infinitesimal symmetry of $L$});
		\item
			The \emph{Noether charge},
			i.e., the map $(x,v)\in TM\mapsto \partial_v L(x,v)[K]\in\mathbb{R}$,
			is a $C^1$ one-form $Q$ on $M$, namely
			\begin{equation}
				\label{eq:noether}
				\partial_vL(x,v)[K]=Q(v);
			\end{equation}
		\item For every $x \in M$, the following equality holds:
			\begin{equation}
				Q(K) = -1\label{QKmeno1}
			\end{equation}
	\end{itemize}
\end{assumption}

\begin{remark}
	In \cite{caponio2023-calcvar},
	the Noether charge was assumed to be an
	\emph{affine}
	function on each tangent space.
	For the sake of simplicity,
	we present the main results under the more restrictive assumption of linearity.
	A discussion about the affine case is given in Appendix~\ref{sec:affineNoether}.
\end{remark}
\begin{remark}
	The $C^3$ regularity condition on $K$ is needed to get
	that a certain map constructed by using the flow of $K$
	is a diffeomorphism
	(see Proposition~\ref{prop:N_pq-Npgammat}).
	We don't know if the regularity of $K$ can be lowered there to $C^1$.
\end{remark}

\begin{assumption}
	\label{ass:L1}
	The Lagrangian $L_c\colon TM \to \mathbb{R}$,
	defined by
	\begin{equation}
		\label{eq:def-Lc}
		L_c(x,v) \coloneqq L(x,v) + Q^2(v),
	\end{equation}
	satisfies the following conditions:
	\begin{itemize}
		\item there exists a continuous function $C\colon M\to(0,+\infty)$
			such that for all $(x,v)\in TM$,
			the following inequalities hold:
			\begin{align}
				\label{eq:pinched}
				L_c(x,v) &\leq C(x)\big(\norm{v}^2+1\big);\\
				\label{eq:partialxpinched}
				|\partial_x L_c(x,v)| &\leq C(x)\big(\norm{v}^2+1\big);\\
				\label{eq:partialvpinched}
				|\partial_v L_c(x,v)| &\leq C(x)\big(\norm{v}+1\big);
			\end{align}
		\item 
there exists a continuous function $\lambda\colon M \to (0,+\infty)$
			such that for each $x\in M$ and for all $v_1, v_2\in T_xM$,
			the following inequality holds:
			\begin{equation}
				\label{eq:monotone}
				\big(\partial_v L_c(x,v_2)-\partial_v L_c(x,v_1)\big)[v_2-v_1]
				\geq \lambda(x)\norm{v_2 - v_1}^2;
			\end{equation}
	\end{itemize}
\end{assumption}

\begin{remark}\label{LcNoether}
	As proven in~\cite[Proposition 2.5]{caponio2023-calcvar},
	$K$ is also an infinitesimal symmetry of $L_c$,
	and a simple computation shows that
	\[
		\partial_v L_c(x,v)[K] = - Q(v).
	\]
\end{remark}

\begin{remark}
	\label{rem:stationary-local-structure}
	From~\cite[Proposition 7.4]{caponio2023-calcvar},
	if $L$ satisfies Assumptions~\ref{ass:L} and \ref{ass:L1},
	it admits a \emph{stationary product type local structure}.
	This means that for each point $p \in M$,
	there exists a neighbourhood $U_p \subset M$,
	an open neighborhood $S_p$ of $\mathbb{R}^{m}$,
	an open interval $I_p$ of $\mathbb{R}$,
	and a diffeomorphism $\phi\colon S_p \times I_p \to U_p$ such that,
	denoting $t$ as the natural coordinate of $I_p$,
	\begin{equation*}
		\label{eq:K-in-local-chart}
		\phi(\partial_t) = K|_{U_p},
	\end{equation*}
	and the function $L$ can be expressed as follows:
	\begin{equation}
		\label{eq:L-localchart}
		L(x,v) = L\circ \phi \big((y,t),(\nu,\tau)\big)
		= L_0(y,\nu) + \omega_y(\nu)\tau - \frac{1}{2}\tau^2,
	\end{equation}
	where
	\begin{itemize}
		\item $(y,t)\in S_p \times I_p$,
			$(\nu,\tau)\in \mathbb{R}^{m}\times\mathbb{R}$,
			and $(x,v) = \phi\big((y,t),(\nu,\tau)\big)$;
		\item $L_0 \in C^1(S_p)$ is a Lagrangian that satisfies
			the growth conditions \eqref{eq:pinched}--\eqref{eq:partialvpinched}
			with respect to the norm $\norm{\cdot}_{S_p}$
			and it is fiberwise strongly convex, i.e., \eqref{eq:monotone} holds 
			(with $L_c$ replaced by $L_0$)
			for some function $\lambda\colon S_p \to (0,+\infty)$;
		\item $\omega_y$ is the $C^1$ one-form induced by $Q$ on $S_p$.
	\end{itemize}
	Using this notation, we have the following equalities:
	\begin{align}
		\label{eq:Q-in-charts}
		Q(v) &= Q\circ \phi(\nu,\tau) = \omega_y(\nu) - \tau;\\
		\label{eq:Lc-in-charts}
		L_c(x,v) &=
		L_c\circ \phi_* \big((y,t),(\nu,\tau)\big)
		= L_0(y,\nu) + \omega_y^2(\nu) - \omega_y(\nu)\tau + \frac{1}{2}\tau^2;\nonumber
	\end{align}
	and
	\begin{equation}
		\label{eq:E-in-charts}
		E(x,v) = E\circ \phi_{*}\big((y,t),(\nu,\tau)\big)
		= E_0(y,\nu) + \omega_y(\nu)\tau - \frac{1}{2}\tau^2,
	\end{equation}
	where
	$E_0(y,\nu) = \partial_\nu L_0(y,\nu)[\nu] - L_0(y,\nu)$.
\end{remark}

\begin{remark}
	\label{rem:E0-min0}
	For every $p \in M$, let $\phi_p\colon S_p \times I_p \to M$
	be a mapping that satisfies~\eqref{eq:L-localchart}.
	Since $L_0$ is fiberwise strongly convex, we can conclude that
	\begin{equation}
		\label{eq:min-E0}
		E_0(y,\nu) > E_0(y,0) = - L_0(y,0),
		\qquad \forall \nu \ne 0,
	\end{equation}
	Indeed, 
	from the strict convexity of $L_0$
	we have
	\[
		L_0(y,0) > L(y,\nu) + \partial_\nu L(y,\nu)[-\nu],
		\qquad \forall \nu \ne 0.
	\]
\end{remark}

\begin{assumption} 
	\label{ass:bounds}
	We require:
	\begin{equation}
		\label{L0}
		\sup_{x\in M} L(x,0)<+\infty. 
	\end{equation}
\end{assumption}
\begin{remark}
	We need the last assumption to guarantee the existence 
	of $\kappa \in \mathbb{R}$ satisfying~\eqref{eq:condition-on-kappa},
	which is a key condition for our main result.
\end{remark}

\subsection{Lorentz-Finsler metrics}
\label{sec:LF}

Provided the existence of an infinitesimal symmetry,
an important kind of Lagrangians that satisfy the above assumptions is 
given by Lorentz-Finsler metrics, introduced by J. K. Beem
in \cite{Beem70}. 
\begin{definition}
	\label{def:Lorentz-Finsler-metric}
	Let $M$ be a smooth, connected manifold of dimension $m+1$.
	A Lagrangian $L_F\colon TM \to \mathbb{R}$ is called 
	\emph{Lorentz-Finsler metric} if it satisfies the following conditions:
	\begin{itemize}
		\item[(a)]$L_F \in C^1(TM) \cap C^2(TM \setminus 0)$,
			where $0$ denotes the zero section of $TM$;
		\item[(b)] $L_F(x,\lambda v) = \lambda^2 L_F(x,v)$, for all $\lambda > 0$;
		\item[(c)] for any $(x,v)\in TM \setminus 0$, the vertical Hessian of $L_F$,
			i.e. the symmetric matrix
			\begin{equation*}
				\label{eq:def_galphabeta}
				(g_F)_{\alpha\beta}(x,v)\coloneqq
				\frac{\partial^2 L_F}{\partial v^\alpha \partial v^\beta}(x,v),
				\quad \alpha,\beta = 0,\dots,m,
			\end{equation*}
			is non-degenerate with index $1$.
	\end{itemize}
\end{definition}
\begin{remark}\label{weakFinsler}
	The regularity conditions required on $L_F$ are sometimes too rigid
	and we relax them to include some interesting classes of Lagrangians
	(see \cite{LaPeHa2012, caponio2018}). 	
	The first and the last conditions above will be replaced  by:
	\begin{itemize}
		\item[(a')]Let $M$ be a smooth, connected,  manifold of dimension $m+1$,
			$m\geq 1$, and  $L_F\in C^{1}(TM)\cap C^2(\mathcal O$),
			where $\mathcal O\subset TM\setminus 0$ is such that 
			$\mathcal O_x\coloneqq\mathcal O\cap T_xM\neq \emptyset $
			for all $x\in M$,
			and $\mathcal O_x$ is an open set in $T_xM$
			which  is a  linear cone
			(i.e. $\lambda v\in \mathcal O_x$, for all $\lambda >0$,
			if $v\in\mathcal O_x$);
			moreover, for any $v_1, v_2\in T_xM$ there exist
			two sequences of vectors $v_{1k}, v_{2k}$ such that,
			for all $k\in\mathbb N$,
			the segment with extreme points $v_{1k}$ and $v_{2k}$
			is entirely contained in $\mathcal O_x$
			and $v_{ik}\to v_i$,  $i=1,2$.
		\item[(c')] Condition (c) is valid  for each $(x,v)\in \mathcal O$;
			moreover  the  eigenvalues $\lambda_i(x,v)$
			of $(g_F)_{\alpha,\beta}(x,v)$
			are bounded away from $0$ on $\mathcal O_x$,
			i.e. there exists $\lambda_+(x)>0$ such that 
			\begin{equation}
				\label{lambdai}
				|\lambda_i(x,v)|\geq \lambda_+(x),
			\end{equation}
			for all $i \in \{0,\ldots ,m\}$ and $v\in \mathcal O_x$.    
	\end{itemize}
\end{remark}

If $L_F$ is a Lorentz-Finsler metric,
then the couple $(M,L_F)$ is called a
\emph{Finsler spacetime}.

The study of the notion of a Finsler spacetime has received renewed impetus
from various sources. 
V. Perlick's work \cite{Perlick2006365}, which explores Fermat's principle,
was particularly influential.
Subsequent contributions came from \cite{GiGoPo07}
(also see \cite{KoStSt09}),
which revived the research initiated by G. Y. Bogoslowsky
\cite{Bogosl77,Bogosl77a,Bogosl94}, and from \cite{GiLiSi07}.
Additional momentum was provided by the works of V. A. Kosteleck\'y
and collaborators \cite{Kostel11,ColMcD12,KoRuTs12,Russel15}, as well as
C. Pfeifer, N. Voicu, and their coworkers
(refer to \cite{PfeWoh11,FusPab16,HohPfe17,Voicu17,FuPaPf18,HoPfVo19} for further details).
Notable mathematical contributions include \cite{JavSan14a, Minguz15, Minguzzi2015}, which
have influenced the field in a different manner.
For a comprehensive historical overview, diverse
definitions of a Finsler spacetime, and additional references, interested
readers are directed to~\cite{GaPiVi12, Minguzzi2016, CapMas20,Javaloyes2020}.

As we will see later, the significance of Lorentz-Finsler metrics
relies on the $2$-homogeneity assumption.
This homogeneity ensures that the solutions
of the Euler-Lagrange equations,
with a suitably prescribed energy value
(in this case, less than or equal to $0$),
connecting a point to a flow line of the infinitesimal symmetry vector field,
are the ones for which the time of arrival is critical.
Therefore, Fermat's principle holds (see Remark~\ref{2hom}).

\begin{proposition}
	\label{prop:LF-satisfiies-assumptions}
	Let $L_F\colon TM \to \mathbb{R}$ be a Lorentz-Finsler metric
	satisfying $(a')$, $(b)$ and $(c')$ above,
	and assume there exists an infinitesimal symmetry  $K\colon M \to TM$
	such that \eqref{eq:noether} and \eqref{QKmeno1} hold.
	Then Assumptions~\ref{ass:L1} and~\ref{ass:bounds} hold.
\end{proposition}
\begin{proof}
	Assumption~\ref{ass:bounds} is ensured by 
	the $2$--homogeneity of $L_F$,
	since $L_F(x,0) = 0$ for every $x \in M$.
	Let us show that Assumption~\ref{ass:L1} holds. As a first step,
	we notice  that the Lagrangian
	$L_c\colon TM \to \mathbb{R}$, defined by
	\[
		L_c(x,v) = L_F(x,v) + Q^2(v),
	\]
	admits vertical Hessian at any
	$(x,v)\in \mathcal O$
	that is a positive definite bilinear form on $T_xM$.
	For any $(x,v) \in \mathcal O$, we 
	have
	\begin{equation}\label{HessLc}
		\partial_{vv}L_c(x,v) = 
		\partial_{vv}L_F(x,v) + 2 Q \otimes Q.
	\end{equation}
	For each $w \in T_x M$,
	we   have, thanks to \eqref{eq:noether},
	\begin{multline}
		\label{eq:Hessian_LF}
		\partial_{vv}L_F(x,v)[K,w]
		=\frac{\partial^2L_F}{\partial s \partial t} 
		(x,v + tK + sw)\Big|_{(s,t) = (0,0)}\\
		= \frac{\partial(\partial_v L_F(x,v + s w)[K])}{\partial s}\Big|_{s = 0}
		= \frac{\partial Q(v + sw)}{\partial s}\Big|_{s = 0}
		= Q(w),
	\end{multline}
	hence we obtain 
	\[
		\partial_{vv}L_c(x,v)[K,K]
		= \partial_{vv}L_F(x,v)[K,K] + 2 Q^2(K)
		= Q(K) + 2 = 1 > 0.
	\]
	Now consider $w \in \mathrm{ker}\, Q$;
	from~\eqref{eq:Hessian_LF} we have 
	$\partial_{vv}L_F(x,v)[w,K] = 0$,
	and since $\partial_{vv}L_F(x,v)$ has index $1$
	we obtain that 
	$\partial_{vv}L_c(x,v)[w,w] = \partial_{vv}L_F(x,v)[w,w]>0$,
	for all $w\in \ker Q$, 
	from which we conclude that 
	$\partial_{vv}L_c(x,v)[\cdot,\cdot]$ is positive definite.

	Let 
	\begin{equation}
		\label{lambda}
		\lambda(x)\colon\inf_{v\in\mathcal O_x}\min_{w\in T_x M, \|w\|=1}
		\partial_{vv}L_c(x,v)[w,w].
	\end{equation}
	As any $w\in \ker Q$ is orthogonal to $K_x$
	with respect to both bilinear forms $\partial_{vv}L_c(x,v)$
	and $\partial_{vv}L_F(x,v)$,
	by  \eqref{HessLc} we deduce that
	the determinants of $(g_F)_{\alpha\beta}$ and $(g_c)_{\alpha\beta}$
	are opposite numbers and then,
	from \eqref{lambdai} we conclude that  $\lambda(x)>0$, for all $x\in M$.

	Inequality~\eqref{eq:monotone} then follows by
	the mean value theorem applied to the function
	$v \in \mathcal O_x \mapsto \partial_vL_c(x,v)[v_2 - v_1] $,
	when $v_1$ and $v_2$ both belong to $\mathcal O_x$
	and the segment having them as extreme points is contained
	in $\mathcal O_x$ as well.
	Then, for each $x \in M$,~\eqref{eq:monotone} follows
	by continuity due to the property of approximation by segments in (a').
	The inequalities~\eqref{eq:pinched},~\eqref{eq:partialxpinched}
	and~\eqref{eq:partialvpinched} are 
	ensured by the fact that $L_c$ is $C^1$ on $TM$ and
	it is positive homogeneous of degree $2$ w.r.t. $v$.

\end{proof}
\begin{remark}
	\label{classicalFinsler}

	As shown in the above proof,
	the vertical Hessian of $L_c$ is positive definite on $\mathcal O$,
	the last being dense in $TM$.
	Hence, by homogeneity,
	$L_c$ is a non-negative fiberwise strongly  function on $TM$.
	Moreover, the vertical Hessian of $F_c\coloneqq \sqrt{L_c}$
	at any $(x,v)\in \mathcal O$
	is positive semi-definite (see, e.g.,~\cite[p.~8]{BaChSh00}).
	Hence, for any $v_1$ and $v_2$ belonging to $\mathcal O_x$
	defining a segment contained in $\mathcal O_x$,
	we get by Taylor's theorem,
	\[
		F_c(x,v_2)\geq F_c(x,v_1)+\partial_{v}F_c(x,v_1)[v_2-v_1].
	\]
	By continuity and the approximation by segments property in (a'),
	the above inequality holds on $TM$,
	hence $F_c$ is fiberwise convex and therefore it is a Finsler metric on $M$,
	(i.e., $F_c(x,\cdot)$ is  non-negative, positively homogeneous,
	and satisfies the triangle inequality on $T_xM$,
	for each $x\in M$)
	whose square is only of class $C^1$ on $TM$.

\end{remark}

As a consequence of Proposition~\ref{prop:LF-satisfiies-assumptions},
if $L_F\colon TM \to M$ is a Lorentz-Finsler metric 
and there exists a complete vector field $K$ such that 
Assumption~\ref{ass:L} holds,
then Remark~\ref{rem:stationary-local-structure}
ensures that $L_F$ can be locally expressed 
as follows:
\begin{equation}
	\label{eq:LF-local}
	L_F(x,v) = L_F\circ \phi((y,t),(\nu,\tau))
	= F^2(y,\nu) + \omega_y(\nu)\tau - \frac{1}{2}\tau^2,
\end{equation}
where $F\colon TS \to \mathbb{R}$ is
a Finsler metric on $S$, with $F^2\in C^1(TM)$.
Whenever  $L_F$ is not twice differentiable only at the line sub-bundle of $TM$ defined by $K$,
$F$  becomes a classical   Finsler metric on $S$,
(i.e.  $F^2 \in C^2(TS\setminus 0)$ and,
for each $y\in S$, $F(y,\cdot)$ is a Minkowski norm on $T_yS$,
see e.g. \cite[\S 1.2]{BaChSh00}).

Since in this case $K$ is a \emph{timelike Killing vector field},
namely it is an infinitesimal symmetry of $L_F$
such that $L_F(x,K) < 0$ for every $x$,
$(M,L_F)$ is called \emph{stationary} Finsler spacetime.
In particular,
if $L_F$  is twice differentiable on $TM\setminus 0$,
then $F^2(y,\cdot)$ in \eqref{eq:LF-local} must be
the square of the norm of a positive definite inner  product on $T_yS$.
We thank the referee for this observation.
In fact, a special kind of stationary Finsler spacetimes are the 
\emph{stationary Lorentzian manifolds}, 
namely those Lorentzian manifolds $(M,g_L)$
for which $g_L$ is a Lorentzian metric and
there exists a timelike Killing vector field for $g_L$.
In this case, the stationary product type local structure 
is given by 
\begin{equation*}
	\label{eq:g_L-local}
	g_L(v,v) = g_R(\nu,\nu)+\omega(\nu)\tau-\frac{1}{2}\tau^2,
\end{equation*} 
where $g_R$ is a  Riemannian metric
on an open neighbourhood $S$ of $\mathbb{R}^m$.
In this direction, the results in this paper improve
previous results about stationary Lorentzian metrics
(see,  \cite{fortunato1995,caponio2002-diffgeo,CaponioMasiello2011}),
since just $C^1$ stationary metrics with a $C^3$
timelike Killing vector field are allowed and both lightlike and
timelike geodesics can be considered in an unified setting.

\section{Variational setting}
\label{sec:fixed-energy-functional}
Let us fix a point $p \in M$
and consider a flow line $\gamma\colon \mathbb{R} \to M$
of $K$ that does not pass through $p$,
i.e., $p \notin \gamma(\mathbb{R})$.
We are interested in finding solutions of the Euler-Lagrange equations
that connect $p$ to points on $\gamma$
with a fixed energy $\kappa \in \mathbb{R}$.
Specifically, we seek to characterize curves $z\in C^1([0,1],M)$
that satisfy~\eqref{eq:Euler-Lagrange},
with $z(0) = p$, $z(1) \in \gamma(\mathbb{R})$,
and $E(z(s),\dot{z}(s)) = \kappa$ for all $s \in [0,1]$.

We define the action functional
$\mathcal{L}\colon H^1([0,1],M) \to \mathbb{R}$ as follows:
\begin{equation*}
	\label{eq:def-calL}
	\mathcal{L}(z) \coloneqq \int_{0}^{1} L(z,\dot{z})\, \d s.
\end{equation*}
Similarly, we define the energy functional:
\begin{equation*}
	\label{eq:def-calE}
	\mathcal{E}(z) \coloneqq \int_{0}^{1} E(z,\dot{z})\, \d s.
\end{equation*}
We note that both $\mathcal{L}$ and $\mathcal{E}$
are well-defined on $H^1([0,1],M)$ and they are respectively a $C^{1}$ and a $C^{0}$ functional  due to \eqref{eq:def-Lc}, the growth conditions \eqref{eq:pinched}--\eqref{eq:partialvpinched} and the fiberwise convexity of $L_c$ \eqref{eq:monotone} (see, e.g., the first part of the proof of Proposition 3.1 in \cite{Abbondandolo2009}). 

\begin{remark}
	\label{rem:energy-noFunctional}
	Henceforth, we will assume that $\mathcal E$ is a $C^1$ functional.
	This holds if $L$ is positively homogeneous of degree $2$
	in the velocities, since in that case $\mathcal E=\mathcal L$;
	moreover it holds if $L_c$ is a $C^2$,
	strongly convex Lagrangian on $TM$ with second derivatives satisfying
	assumptions (L1') in \cite[p. 605]{Abbondandolo2009}.
\end{remark}

Recalling that we have chosen a fixed point $p \in M$,
we define the set $\Omega_{p,r}(M)$ for every $r \in M$ as follows:
\[
	\Omega_{p,r}(M) \coloneqq \big\{ z \in H^1([0,1],M): z(0)=p, z(1) = r \big\},
\]
and we denote by $\mathcal{L}_{p,r}$ the restriction of $\mathcal{L}$ to $\Omega_{p,r}(M)$.

\begin{remark}
	\label{rem:E-constant-on-critical-points}
	According to \cite[Proposition A.1]{caponio2023-calcvar},
	if $z$ is a critical point of $\mathcal{L}_{p,r}$,
	then both $z$ and the function
	\[
		s \mapsto \partial_vL(z(s),\dot{z}(s))[\dot{z}(s)]
	\]
	are of class $C^1$.
	As a consequence, $z$ is a critical point of $\mathcal{L}_{p,r}$
	if and only if equation \eqref{eq:Euler-Lagrange}
	holds and there exists $\kappa \in \mathbb{R}$
	such that equation \eqref{eq:E-const} holds.
\end{remark}

\subsection{Preliminary results}
\label{sec:preliminary}

Recalling that $K$ is a complete vector field, we denote by $\psi\colon
\mathbb{R} \times M \to M$ the flow of $K$,
and by $\partial_u\psi$ and $\partial_x\psi$
the partial derivatives of $\psi(t,x)$ with respect to $t \in \mathbb{R}$ and $x \in M$,
respectively. 

Let us denote by $K^c$ the complete lift of $K$ to $TM$ (see, e.g., \cite{caponio2018}).
Then, for any $(x,v)\in TM$, the  flow $\psi^c$ of $K^c$ on $TM$ is given by
$\psi^c(t,x,v)=\big(\psi(t,x),\partial_x\psi(t,x)[v]\big)$,
and we have
\[
	K^c(L)\big(\psi^c(t,x,v)\big)=
	\dfrac{\partial \big (L\circ \psi^c\big)}{\partial t}(t,x,v).
\]
Since $K$ is an infinitesimal symmetry of $L$, we have
\begin{equation}
	\label{eq:L-Kinvariant}
	\dfrac{\partial \big (L\circ \psi^c\big)}{\partial t}(t,x,v)=0,
\end{equation}
which implies
\begin{equation}
	\label{eq:Killing-LocCoord}
	K^c(L)(x,v)
	= \frac{\partial L }{\partial x^h}(x,v)K^h(x)
	+ \frac{\partial L}{\partial v^h}(x,v)\frac{\partial K^h}{\partial x^i}(x)v^i =0.
\end{equation}
Moreover, from \eqref{eq:L-Kinvariant} we also obtain
\begin{equation}
	\label{eq:L-psiInvariant}
	L(x,v) = L\big(\psi(t,x),\partial_x\psi(t,x)[v]\big),
	\quad \forall (x,v) \in TM,\, t \in \mathbb{R},
\end{equation}
and consequently
\begin{equation}
	\label{eq:dvL-psiInvariant}
	\partial_v L(x,v)[\xi]
	= \partial_vL\big(\psi(t,x),\partial_x\psi(t,x)[v]\big)\big[\partial_x\psi(t,x)[\xi]\big].
\end{equation}

\begin{lemma}
	\label{lem:constant-Noether-charge}
	If $z\colon[0,1]\to M$ is a weak solution of
	the Euler-Lagrange equation~\eqref{eq:Euler-Lagrange}
	$($i.e. a critical point of $\mathcal{L}$ on $\Omega_{z(0),z(1)}(M)$$)$,
	then it is a $C^1$ curve and its Noether charge is constant,
	namely there exists $c \in \mathbb{R}$ such that
	\begin{equation*}
		\label{eq:const-Noether-charge}
		\partial_vL(z(s),\dot{z}(s))[K(z(s))] = c,
		\quad \forall s \in [0,1].
	\end{equation*}
\end{lemma}
\begin{proof}
	By~\cite[Proposition A.1]{caponio2023-calcvar},
	both $z$ and $\partial_vL(z,\dot{z})$ are of class $C^1$.
	Therefore, it suffices to prove that, for every $s \in [0,1]$,
	we have
	\[
		\frac{\d}{\d s}\Big(
		\partial_vL(z(s),\dot{z}(s))[K(z(s))]\Big) = 0.
	\]
	Therefore, we can work on a local coordinate system
	$(x^0,\dots,x^{m},v^0,\dots,v^m)$ of $TM$
	and, using~\eqref{eq:Euler-Lagrange} and~\eqref{eq:Killing-LocCoord},
	we obtain the following chain of equalities:
	\begin{align*}
		\lefteqn{\frac{\d  }{\d s}\left(\frac{\partial L}{\partial v^i}\big(z(s),\dot{z}(s)\big)K^i(z(s))\right)}&\\
		&\quad=\frac{\d }{\d s}\left(\frac{\partial L}{\partial v^i}\big(z(s),\dot{z}(s)\big)\right)K^i(z(s))
		+\frac{\partial L}{\partial v^i}\big(z(s),\dot{z}(s)\big)\frac{\partial K^i}{\partial x^h}(z(s))\dot z^h(s)\\
		&\quad=\frac{\partial L}{\partial x^i}\big(z(s),\dot{z}(s)\big)K^i(z(s))
	+\frac{\partial L}{\partial v^i}\big(z(s),\dot{z}(s)\big)\frac{\partial K^i}{\partial x^h}(z(s))\dot z^h(s) =0.
																																				\qedhere
	\end{align*}
\end{proof}
On the basis of Lemma~\ref{lem:constant-Noether-charge},
the curves with a constant Noether charge
are the only ones that can be critical points of the action functional.
The following results ensure that this subset of curves is
indeed a closed manifold of class $C^1$,
allowing for a simplification of the
variational setting by considering only these curves.
A detailed proof can be found in~\cite{caponio2023-calcvar}
and relies on the linearity assumption of the Noether charge.

Let us define the following sets:
\begin{equation*}
	\label{eq:def_Npq}
	\mathcal{N}_{p,r} \coloneqq \{z \in \Omega_{p,r}(M): \exists c \in \mathbb{R} \text{ such that } Q(\dot{z}) = c \text{ a.e. on } [0,1]\} \subset \Omega_{p,r}(M),
\end{equation*}
and
\begin{multline*}
	\label{eq:def-calWz}
	\mathcal{W}_z := \big\{
		\eta \in T_z\Omega_{p,r}(M): \exists \mu\in H^{1}_0([0,1],\mathbb{R})\\
		\text{ such that } \eta(s)=\mu(s)K(z(s)), \text{ a.e. on } [0,1]
	\big\}.
\end{multline*}

Since $L$ is invariant under the one-parameter
group of local $C^1$ diffeomorphisms
generated by $K$, we have the following result.
\begin{proposition}
	\label{prop:N_pq-characterization}
	The space $\mathcal{N}_{p,r}$ is non-empty,
	it is a $C^1$ closed submanifold of $\Omega_{p,r}(M)$ and satisfies
	\begin{equation}
		\label{eq:N_pq-characterization}
		\mathcal{N}_{p,r} = \left\{
			z \in \Omega_{p,r}(M): \d \mathcal{L}_{p,r}(z)[\eta] = 0, \forall \eta \in \mathcal{W}_z
		\right\}.
	\end{equation}
	Moreover, for every $z \in \mathcal{N}_{p,r}$,
	the tangent space of $\mathcal{N}_{p,r}$ at $z$ is given by
	\begin{equation}
		\label{eq:TzNpq}
		T_z\mathcal{N}_{p,r} = \left\{
			\xi\in T_z\Omega_{p,r}(M): \exists c\in \mathbb{R} \text{ such that }
			\partial_xQ(\dot{z})[\xi]+Q(\dot \xi)=c \text{ a.e.}
		\right\},
	\end{equation}
	and
	\begin{equation}
		\label{eq:TzOmega-splitting}
		T_z\Omega_{p,r}(M) = T_z\mathcal{N}_{p,r} +\mathcal{W}_z.
	\end{equation}
\end{proposition}

\begin{proof}
	The fact that $\mathcal N_{p,r}\neq \emptyset$, for all $p, r\in M$,
	follows from \cite[proposition 6.4]{caponio2023-calcvar}.
	Equality~\eqref{eq:N_pq-characterization} is proved
	in~\cite[Proposition 4.2]{caponio2023-calcvar},
	and \eqref{eq:TzNpq} is a particular case
	of~\cite[Proposition 4.3]{caponio2023-calcvar}.
	Finally, \eqref{eq:TzOmega-splitting} comes 
	from \cite[Lemma 4.4]{caponio2023-calcvar}
	\footnote{
		We would like to draw attention to a misprint
		in~\cite[Lemma 4.4]{caponio2023-calcvar},
		where we note that the ``direct sum''
		should be corrected to ``sum'' as it appears there.
	}. 
\end{proof}

The above result gives the following variational principle
for the critical points of $\mathcal{L}_{p,r}$,
which extends a result by F. Giannoni and P. Piccione
(see \cite{giannoni1999}).

\begin{proposition}
	\label{prop:restriction-Npq}
	Let $\mathcal{J}_{p,r}\colon \mathcal{N}_{p,r} \to \mathbb{R}$
	be the restriction of $\mathcal{L}_{p,r}$ to $\mathcal{N}_{p,r}$.
	Then, $z$ is a critical point for $\mathcal{L}_{p,r}$
	if and only if $z \in \mathcal{N}_{p,r}$ and $z$
	is a critical point for $\mathcal{J}_{p,r}$.
\end{proposition}

\begin{proof}
	See \cite[Theorem 4.7]{caponio2023-calcvar}.
\end{proof}

\section{The variational structure of the action in relation with the flow of \texorpdfstring{$K$}{K}}
\label{sec:flowK}
In this section, we consider the flow of the complete vector field $K$ and its
relationship with the variational structure of the action.
More precisely, let $\psi\colon \mathbb{R} \times M \to M$ denote the flow generated by the vector field $K$.
Given a flow line $\gamma\colon \mathbb{R} \to M$ of $K$,
there exists a point $q \in M$ such that $\gamma(t) = \psi(t,q)$.

Our goal is to prove that for each $t \in \mathbb{R}$, there is a
diffeomorphism between $\mathcal{N}_{p,q}$ and $\mathcal{N}_{p,\gamma(t)}$.
This enables us to define a functional (see \eqref{eq:hk}) on
$\mathcal{N}_{p,q}\times \mathbb{R}$
and obtain an alternative equation for
solutions of the Euler-Lagrange equations connecting $p$ and $\gamma$
(see \eqref{eq:EL-hkappa}).
Furthermore, recalling that we seek the solutions of Euler-Lagrange equations with a fixed energy $\kappa$,
we show that for any $z\in \mathcal{N}_{p,q}$, there are two values of $t\in \mathbb{R}$ such that $\mathcal{E}(z^t) = \kappa$, where  $\kappa$ satisfies  \eqref{eq:condition-on-kappa} and  $z^t\in\mathcal{N}_{p,\gamma(t)}$ is the curve corresponding to $z$ via the diffeomorphism.
Therefore, we can simplify the problem and study
a couple of functionals defined only on $\mathcal{N}_{p,q}$
(see \eqref{eq:def_t+-}).

Let us define the map
$F^t\colon \Omega_{p,q}(M) \to \Omega_{p,\gamma(t)}(M)$ as follows:
\begin{equation}
	\label{eq:def-Ft}
	\big(F^t(z)\big)(s) \coloneqq \psi(ts,z(s)).
\end{equation}
To simplify the notation, we write
\[
	z^t = F^t(z)
\]
for any $z \in \Omega_{p,q}(M)$.
\begin{proposition}
	\label{prop:N_pq-Npgammat}
	The map $F^t$ is a diffeomorphism with its inverse being $F^{-t}$.
	Furthermore, $F^t|_{\mathcal{N}_{p,q}}$
	is a diffeomorphism from $\mathcal{N}_{p,q}$ to $\mathcal{N}_{p,\gamma(t)}$.
	Therefore, for every $z \in \Omega_{p,q}(M)$,
	we have the following equivalences:
	\begin{equation}
		\label{eq:dFt-Npq}
		\d F^t(z)[\xi] \in T_{z^t}\mathcal{N}_{p,\gamma(t)}
		\quad\text{if and only if}
		\quad \xi \in T_z\mathcal{N}_{p,q},
	\end{equation}
	and
	\begin{equation}
		\label{eq:dFt-Wz}
		\d F^t(z)[\eta] \in \mathcal{W}_{z^t}
		\quad\text{if and only if}
		\quad \eta \in \mathcal{W}_{z}.
	\end{equation}
\end{proposition}
\begin{proof}
	By utilizing a result by R. Palais~\cite{palais1963-Topology}
	and considering that the flow of $K$ is $C^3$,
	we can conclude that $F^t$ is a diffeomorphism
	(cf.~\cite[Proposition 2.2]{caponio2002-diffgeo}).
	Recalling that $\partial_u\psi$ is the differential of $\psi$
	with respect to the first variable,
	we can derive the following equalities:
	\begin{equation*}
		\label{eq:partial_t_psi}
		\partial_u \psi(ts,z(s))[1] = K(\psi(ts,z(s))),
	\end{equation*}
	and 
	\begin{equation}
		\label{eq:partial_x_psi}
		\partial_x\psi(ts,z(s))[K(z(s))] = K(\psi(ts,z(s))).
	\end{equation}
	Consequently, we obtain the velocity of $z^t$ as:
	\begin{equation}
		\label{eq:zt-velocity}
		\frac{\d}{\d s}z^t(s) = \dot{z}^t(s)
		= \partial_u\psi(ts,z(s))[t] + \partial_x\psi(ts,z(s))[\dot{z}(s)].
	\end{equation}
	Now, considering that $Q(K) \equiv -1$, we deduce:
	\begin{equation*}
		\label{eq:dvLzt1}
		\partial_vL(z^t,\dot{z}^t)[K(z^t)] = Q(\dot{z}^t)
		= -t + Q\left(\partial_x \psi(ts,z(s))[\dot{z}(s)]\right).
	\end{equation*}
	Hence, from \eqref{eq:dvL-psiInvariant}, we have:
	\begin{equation}
		\label{eq:dvLzt2}
		Q(\dot{z}^t) = Q(\dot{z}) - t,
	\end{equation}
	which implies that $z^t \in \mathcal{N}_{p,\gamma(t)}$
	if and only if $z \in \mathcal{N}_{p,q}$.
	Therefore, this implies \eqref{eq:dFt-Npq}.
	Finally, \eqref{eq:dFt-Wz} follows from
	$\d F^t(z)[\nu]=\partial_x\psi(ts,z(s))[\nu(s)]$
	and~\eqref{eq:partial_x_psi}.
\end{proof}

We introduce the functional
$\mathcal{H}_{p,q}\colon \Omega_{p,q}(M)\times \mathbb{R} \to \mathbb{R}$
defined as follows:
\begin{equation}
	\label{eq:def-hk}
	\mathcal{H}_{p,q}(z,t) \coloneqq \mathcal{L}_{p,\gamma(t)}(F^t(z)).
\end{equation}
Using \eqref{eq:zt-velocity} and observing that
$\partial_u\psi(ts,z(s))[t]=t\partial_x\psi(ts,z(s))[K(z(s)]$,
we can deduce the expression:
\begin{equation}
	\label{eq:dz^t}
	\dot{z}^t = \partial_x\psi(ts,z(s))\big[\dot{z} + tK(z(s))\big],
\end{equation}
so that, by applying also \eqref{eq:L-psiInvariant},
we can rewrite $\mathcal{H}_{p,q}(z,t)$ as
\begin{equation}
	\label{eq:hk}
	\mathcal{H}_{p,q}(z,t) = \int_{0}^{1}L\big(z,\dot{z}+ t K(z)\big)\d s.
\end{equation}

Considering that $F^t|_{\mathcal{N}_{p,q}}$ is a diffeomorphism,
we obtain the following result,
which allows us to focus our study on critical curves of
$\mathcal{H}_{p,q}$ within $\mathcal{N}_{p,q}$.
\begin{proposition}
	\label{prop:restriction-Hpq}
	For $(z,t) \in \Omega_{p,q} (M)\times \mathbb{R}$, the following statements hold:
	\begin{equation}
		\label{eq:partialz-H-0}
		\partial_z\mathcal{H}_{p,q}(z,t)[\xi] = 0,
		\qquad \forall \xi \in T_z\Omega_{p,q}(M),
	\end{equation}
	if and only if $z \in \mathcal{N}_{p,q}$ and
	\begin{equation}
		\label{eq:partialz-H-1}
		\partial_z\mathcal{H}_{p,q}(z,t)[\xi] = 0,
		\qquad \forall \xi \in T_z\mathcal{N}_{p,q}.
	\end{equation}
\end{proposition}
\begin{proof}
	If \eqref{eq:partialz-H-0} holds, we can use \eqref{eq:def-hk}
	and Proposition \ref{prop:N_pq-Npgammat}
	to conclude that $z^t=F^t(z)$ is a critical point of $\mathcal{L}_{p,\gamma(t)}$,
	and by Proposition \ref{prop:restriction-Npq},
	$z^t$ belongs to $\mathcal N_{p,\gamma(t)}$.
	Consequently, we have $z=F^{-t}(z^t)\in \mathcal N_{p,q}$,
	and \eqref{eq:partialz-H-1} trivially follows from \eqref{eq:partialz-H-0}.

	For the other implication, we need to show that if $z \in \mathcal{N}_{p,q}$, then
	\begin{equation}
		\label{eq:partialz-H-2}
		\partial_z\mathcal{H}_{p,q}(z,t)[\eta] = 0,
		\qquad \forall \eta \in \mathcal{W}_{z}.
	\end{equation}
	By contradiction,
	let's assume that $z\in \mathcal{N}_{p,q}$ and \eqref{eq:partialz-H-2} does not hold.
	According to the definition of $\mathcal{H}_{p,q}$,
	there exists $\eta \in\mathcal{W}_z$ such that
	\begin{equation*}
		\label{eq:partialz-H-3}
		\partial_z\mathcal{H}_{p,q}(z,t)[\eta]
		= \d\mathcal{L}_{p,\gamma(t)}(F^t(z))\big[ \d F^t(z)[\eta]\big]
		\ne 0.
	\end{equation*}
	Using \eqref{eq:dFt-Wz}, we know that $\d F^t(z)[\eta] \in \mathcal{W}_{z^t}$.
	Applying Proposition~\ref{prop:N_pq-characterization},
	we can conclude that $F^t(z) \notin \mathcal{N}_{p,\gamma(t)}$,
	which contradicts Proposition~\ref{prop:N_pq-Npgammat}.
\end{proof}

\begin{corollary}
	\label{cor:partial_zH}
	If $(z,t)$ satisfies~\eqref{eq:partialz-H-0},
	then $z^t$ is a critical point for $\mathcal{L}_{p,\gamma(t)}$.
	The following Euler-Lagrange equations $($in local coordinates$)$ hold:
	\begin{multline}
		\label{eq:EL-hkappa}
		\frac{\partial L}{\partial x^i}\big(z,\dot{z} + tK(z)\big)
		-\frac{\d}{\d s}\frac{\partial L}{\partial v^i}\big(z,\dot{z}+ tK(z)\big)\\
		+ t\frac{\partial L}{\partial v^j}\big(z,\dot{z} + tK(z)\big)\frac{\partial K^j}{\partial x^i}(z)
		= 0,
		\quad \forall s \in [0,1],
	\end{multline}
	and there exists $\kappa\in \mathbb{R}$ such that
	\begin{equation}
		\label{eq:energylaw-hkappa}
		E\big(z,\dot{z} + tK(z)\big) = \kappa,
		\quad \forall s \in [0,1].
	\end{equation}
\end{corollary}
\begin{proof}
	According to Proposition \ref{prop:restriction-Hpq},
	if \eqref{eq:partialz-H-0} holds,
	then \eqref{eq:EL-hkappa} is an immediate consequence of \eqref{eq:hk}
	and the du Bois-Reymond lemma.
	By \eqref{eq:def-hk}, $z^t=F^t(z)$ is a critical point
	of $\mathcal{L}$ on $\Omega_{p,\gamma(t)}(M)$.
	Hence, using Remark \ref{rem:E-constant-on-critical-points},
	we can conclude that there exists a constant $\kappa$ such that $E(z^t, \dot z^t) = \kappa$.
	Combining
	\eqref{eq:L-psiInvariant}, \eqref{eq:dvL-psiInvariant}, and \eqref{eq:dz^t},
	we obtain \eqref{eq:energylaw-hkappa}.
\end{proof}
\begin{proposition}
	\label{prop:LandEofzt}
	For every $(x,v) \in TM$ and every $t \in \mathbb{R}$,
	the following two equations hold:
	\begin{equation}
		\label{eq:L-zt}
		L\big(x,v + tK(x)\big) = L\big(x,v\big) + t Q(v) - \frac{1}{2}t^2,
	\end{equation}
	and
	\begin{equation}
		\label{eq:E-zt}
		E\big(x,v + tK(x)\big) = E\big(x,v\big) + t Q(v) - \frac{1}{2}t^2.
	\end{equation}
	As a consequence, for every $(x,v)\in TM$, we have
	\begin{equation}
		\label{eq:diff-LandE}
		L\big(x,v + tK(x)\big) - E\big(x,v + tK(x)\big) = L\big(x,v\big) - E\big(x,v\big).
	\end{equation}
\end{proposition}
\begin{proof}
	We will prove~\eqref{eq:L-zt};
	the computations for~\eqref{eq:E-zt} are analogous.
	Since the result has a local nature, we can use~\eqref{eq:L-localchart}.
	For every $(x,v) \in TM$, we can write
	\begin{equation*}
		\begin{split}
			L(x,v + tK)
	  &= L\circ \phi_* \big((y,t),(\nu,\tau + t)\big) \\
	  & = L_0(y,\nu) + \omega_y(\nu)(\tau + t) - \frac{1}{2}(\tau + t)^2 \\
	  & = \Big(L_0(y,\nu) + \omega_y(\nu)\tau - \frac{1}{2}\tau^2\Big)
	  + \big(\omega_y(\nu) - \tau\big)t
	  - \frac{1}{2}t^2\\
	  & = L(x,v) + t Q(v) - \frac{1}{2}t^2.
		\end{split}
	\end{equation*}
	This completes the proof.
\end{proof}

Using \eqref{eq:L-zt} and recalling that $Q(\dot{z})$
is constant for all $z \in \mathcal{N}_{p,q}$,
the functional $\mathcal{H}_{p,q}$ can be written as:
\begin{equation}
	\label{eq:calH-splitting}
	\mathcal{H}_{p,q}(z,t)
	= \int_{0}^{1}L\big(z,\dot{z}\big)\d s + tQ(\dot{z}) - \frac{1}{2}t^2
	= \mathcal{L}(z) +t Q(\dot{z}) - \frac{1}{2}t^2.
\end{equation}

\begin{proposition}
	\label{prop:def-t_+-}
	Let 
	\begin{equation}
		\label{eq:condition-on-kappa}
		\kappa \le - \sup_{x \in M}L(x,0)
	\end{equation}
	$($recall \eqref{L0}$)$.
	Then the functionals
	$t_+^{\kappa},t_-^{\kappa}\colon \mathcal{N}_{p,q} \to \mathbb{R}$
	defined by
	\begin{equation}
		\label{eq:def_t+-}
		t_{\pm}^{\kappa}(z) = Q(\dot{z}) \pm \sqrt{Q^2(\dot{z})
		+ 2\big( \mathcal{E}(z)- \kappa \big)},
	\end{equation}
	are well-defined ,
	and they satisfy the following equation:
	\begin{equation}
		\label{eq:prop_t+-}
		\mathcal{E}(F^{t_{\pm}^{\kappa}(z)}(z)) = \kappa.
	\end{equation}
\end{proposition}
\begin{proof}
	Since $Q(\dot{z})$ is constant for every $z \in\mathcal{N}_{p,q}$,
	from~\eqref{eq:E-zt} we have that $t_{\pm}^{\kappa}(z)$ are the only two solutions of
	\begin{equation*}
		\label{eq:E-zt-equalkappa}
		\mathcal{E}(F^t(z)) = \mathcal{E}(z^t) = \mathcal{E}(z) + t Q(\dot{z}) - \frac{1}{2}t^2 = \kappa.
	\end{equation*}
	Hence, it remains to prove that for every $z \in\mathcal{N}_{p,q}$, we have
	\[
		\mathcal{E}(z) + \frac{1}{2}Q^2(\dot{z}) \ge  \kappa,
	\]
	provided that $\kappa$ satisfies~\eqref{eq:condition-on-kappa}.
	As a consequence, it suffices to prove that
	\begin{equation}
		\label{eq:Ec-Q2gekappa}
		E(x,v) + \frac{1}{2}Q^2(v) \ge \kappa, \qquad \forall (x,v) \in TM.
	\end{equation}
	Using the expression of $L$ in a local chart in a neighbourhood of $x\in M$,
	in particular~\eqref{eq:Q-in-charts} and~\eqref{eq:E-in-charts},
	and setting $(x,v) = \phi_*\big((y,t),(\nu,\tau)\big)$,
	we obtain the following equalities:
	\begin{multline}
		\label{eq:Ec-Q2}
		E(x,v) + \frac{1}{2}Q^2(v)
		= E_0(y,\nu) + \omega_y(\nu)\tau- \frac{1}{2}\tau^2
		+ \frac{1}{2}\big(\omega_y(\nu) - \tau\big)^2\\
		= E_0(y,\nu) + \frac{1}{2}\omega_y^2(\nu),
	\end{multline}
	where $E_0(y,\nu)$ is the energy function of the Lagrangian $L_0$.
	As a consequence, using~\eqref{eq:min-E0}, we obtain
	\begin{equation*}
		\label{eq:tmp2}
		E_0(y,\nu) + \frac{1}{2}\omega_y^2(\nu) \ge E_0(y,0)
		= - L_0(y,0) = - L(x,0).
	\end{equation*}
	Since $\kappa$ satisfies~\eqref{eq:condition-on-kappa}, we infer
	\begin{equation*}
		E(x,v) + \frac{1}{2}Q^2(v) \ge - L(x,0) \ge \kappa, \qquad \forall (x,v) \in TM,
	\end{equation*}
	and we are done.
\end{proof}
\begin{remark}
	\label{kleq0}
	Our problem naturally leads to the
	condition~\eqref{eq:condition-on-kappa}.
	For a Finsler spacetime $(M,L)$
	(see Section~\ref{sec:LF}),
	this condition means $\kappa \le 0$.
	Therefore, we only consider the energy values
	that correspond to causal geodesics
	(timelike or lightlike geodesics). 	
\end{remark}

\begin{lemma}\label{neqkappa}
	If $\kappa$ satisfies \eqref{eq:condition-on-kappa} then 
	\[
		\mathcal E(z)	+\frac 1 2 Q^2(z)>\kappa,
		\quad \forall z\in\mathcal N_{p,q}.
	\]
\end{lemma}
\begin{proof}
	From \eqref{eq:Ec-Q2gekappa}, it is enough to prove that 	
	\[
		\label{eq:tmp3}
		\mathcal{E}(z) + \frac{1}{2}Q^2(\dot{z}) \neq  \kappa.
	\]
	By contradiction assume that $\mathcal{E}(z) + \frac{1}{2}Q^2(\dot{z}) =  \kappa$. 
	Using~\eqref{eq:Ec-Q2} and~\eqref{eq:min-E0},
	we conclude that in any neighbourhood $U_{z(\bar s)}$, $\bar s\in [0,1]$,  as in Remark~\ref{rem:stationary-local-structure},
	and  for a.e. $s$ in a neighbourhood  of $\bar s$, the vector $\dot z(s)$  corresponds
	through $\phi_*$ to a  vector whose component in $TS_{z(\bar s)}$ vanishes.
	This is equivalent to the existence of a function
	$\alpha\colon [0,1] \to \mathbb{R}$
	such that
	\[
		\dot{z}(s) = \alpha(s)K(z(s)),
		\qquad \text{for a.e. } s \in [0,1].
	\]
	Since $Q(\dot z)$ is constant a.e. and $Q(\alpha(s)K(z(s)))=-\alpha(s)$,
	we deduce that $\alpha$ is constant a.e. and $\dot z$ is equivalent to a continuous
	$TM$-valued function on $[0,1]$.
	Hence $p$ and $q$ are on the same flow line of $K$,
	which is a contradiction. 
\end{proof}		
\begin{remark}\label{C1}
	As a consequence of Lemma~\ref{neqkappa},
	$t_{\pm}^{\kappa}$ in \eqref{eq:def_t+-} are $C^1$
	functionals on $\mathcal N_{p,q}$.
\end{remark}

\begin{corollary}
	\label{cor:partial_tH-notnull}
	If $\kappa$ satisfies \eqref{eq:condition-on-kappa}, then
	\[	
		\partial_t\mathcal{H}_{p,q}(z,t_+^{\kappa}(z)) \ne 0,
		\quad \forall z \in \mathcal{N}_{p,q},
	\]
	and the same holds
	replacing $t_+^{\kappa}(z)$ with $t_-^{\kappa}(z)$.
	\end{corollary}   \begin{proof}
	By~\eqref{eq:calH-splitting} and~\eqref{eq:def_t+-}, we have
	\begin{equation}
		\label{eq:computatoin-partial_tH}
		\partial_t\mathcal{H}_{p,q}(z,t_+^{\kappa}(z))
		= Q(\dot{z}) - t_+^{\kappa}(z) = -
		\sqrt{
			Q^2(\dot{z}) + 2\big( \mathcal{E}(z)- \kappa \big)
		}.
	\end{equation}
	Then, the thesis follows by Lemma~\ref{neqkappa}.	
\end{proof}                                       
\section{Main result}
We are ready to proof our main result:
\begin{theorem}
	\label{theorem:almostFermat}
	Let $L\colon TM\to \mathbb{R}$ satisfy
	Assumptions~\ref{ass:L}, \ref{ass:L1}, and \ref{ass:bounds},
	and let $\kappa \in \mathbb{R}$ satisfy~\eqref{eq:condition-on-kappa}.
	A curve $\ell\colon [0,1] \to M$ is a solution of the Euler-Lagrange
	equations~\eqref{eq:Euler-Lagrange}
	joining $p$ and $\gamma$ with energy $\kappa$
	if and only if
	there exists $z \in \mathcal{N}_{p,q}$ such that
	$\ell = F^{t_+^{\kappa}(z)}(z)$ or $\ell = F^{t_-^{\kappa}(z)}(z)$,
	and the following equality holds:
	\begin{equation}
		\label{eq:t_+-critical?}
		\d t_{\pm}^{\kappa}(z)
		= \frac{\d \mathcal{E}(z) - \d \mathcal{L}(z)}
		{\sqrt{Q^2(\dot{z})+ 2\big( \mathcal{E}(z)- \kappa \big)}},
	\end{equation}
	or
	\begin{equation}
		\label{eq:t_--critical?}
		\d t_-^{\kappa}(z)
		= \frac{\d \mathcal{L}(z) - \d \mathcal{E}(z)}
		{\sqrt{Q^2(\dot{z})+ 2\big( \mathcal{E}(z)- \kappa \big)}}.
	\end{equation}
\end{theorem}
\begin{proof}
	Consider a critical curve $\ell \in \mathcal{N}_{p,\gamma(t)}$ with energy $\kappa$.
	We know that $F^t$ is a diffeomorphism,
	so there exists $z \in \mathcal{N}_{p,q}$ such that $F^{-t}(\ell) = z$
	and $t = t_+^{\kappa}(z)$ or $t = t_-^{\kappa}(z)$.
	For this proof, we will focus on the case where $t = t_+^{\kappa}(z)$.

	Since $\ell$ is a critical curve for $\mathcal{L}_{p,\gamma(t)}$,
	by the definition of $\mathcal{H}_{p,q}$ (see~\eqref{eq:def-hk}),
	we have $\partial_z\mathcal{H}_{p,q}(z,t) = 0$.
	Furthermore, using~\eqref{eq:diff-LandE} and the definition of $t_+^{\kappa}(z)$,
	we obtain the following equation:
	\begin{equation}
		\label{eq:H-in-t+}
		\begin{split}
			\mathcal{H}_{p,q}(z,t_+^{\kappa}(z))
	& = \int_{0}^{1} L\big(z,\dot{z} + t_+^{\kappa}(z)K(z)\big) \d s\\
	& = \mathcal{L}(z) - \mathcal{E}(z) + k,
	\qquad \forall z \in \mathcal{N}_{p,q}.
		\end{split}
	\end{equation}
	By differentiating both sides of~\eqref{eq:H-in-t+}, we obtain:
	\begin{equation}
		\label{eq:diff-H-in-t+}
		\partial_z\mathcal{H}_{p,q}(z,t_+^{\kappa}(z))
		+ \partial_t\mathcal{H}_{p,q}(z,t_+^{\kappa}(z))\, \d t_+^{\kappa}(z)
		= \d \mathcal{L}(z) - \d\mathcal{E}(z).
	\end{equation}
	Since we know that $\partial_z\mathcal{H}_{p,q}(z,t_+^{\kappa}(z)) = 0$,
	substituting this into~\eqref{eq:diff-H-in-t+}, we get:
	\[
		\partial_t\mathcal{H}_{p,q}(z,t_+^{\kappa}(z))\, \d t_+^{\kappa}(z)
		= \d \mathcal{L}(z) - \d\mathcal{E}(z).
	\]
	According to Corollary~\ref{cor:partial_tH-notnull},
	we have $\partial_t\mathcal{H}_{p,q}(z,t_+^{\kappa}(z)) \ne 0$
	for every $z \in \mathcal{N}_{p,q}$.
	Using equation~\eqref{eq:computatoin-partial_tH},
	we obtain~\eqref{eq:t_+-critical?}.

	For the converse,
	if $z \in \mathcal{N}_{p,q}$ satisfies~\eqref{eq:t_+-critical?},
	we can use~\eqref{eq:diff-H-in-t+}
	to conclude that $\partial_z\mathcal{H}_{p,q}(z,t_+^{\kappa}(z)) = 0$.
	By Proposition~\ref{prop:restriction-Hpq} and Corollary~\ref{cor:partial_zH},
	we then deduce that $\ell = F^{t_+^{\kappa}(z)}(z)$
	is a critical point of $\mathcal{L}_{p,\gamma(t)}$.
	Hence the thesis follows from Proposition~\ref{prop:restriction-Npq} 
	and Remark~\ref{rem:E-constant-on-critical-points}.
\end{proof}
\begin{remark}\label{2hom}
	If $L$ is homogeneous of degree $2$
	in the velocities
	(i.e., $L$ is a Lorentz-Finsler metric,   Definition~\ref{def:Lorentz-Finsler-metric}),
	then $\mathcal{L}(z) = \mathcal{E}(z)$ for every $z$ and, consequently,
	\eqref{eq:t_+-critical?} and~\eqref{eq:t_--critical?}
	are equivalent to $\d t_+^{\kappa}(z) = 0$ and $\d t_-^{\kappa}(z) = 0$,
	respectively.
	Hence, in this case we re-obtain, for $\kappa\le 0$,
	the Fermat's principle in a stationary spacetime
	that globally splits \cite{fortunato1995}
	(also known as {\em standard stationary spacetime})
	as well as in a
	stationary spacetime that may  not globally split 
	\cite{caponio2002-diffgeo}.
	Furthermore, we also obtain a Fermat's principle in a stationary Finsler
	spacetime that is not necessarily a stationary splitting one
	(compare with \cite[Appendix B]{caponio2018}),
	including  also timelike geodesics.
\end{remark}

\section{An existence and multiplicity result}
\label{Section:ex&mult}
In this section we assume that $L$ is a Lorentz-Finsler metric as in
Section~\ref{sec:LF}, satisfying Assumption~\ref{ass:L}. 
By Theorem~\ref{theorem:almostFermat} and Remark~\ref{2hom},
the critical points $z$ of the functionals
$t^\kappa_\pm\colon \mathcal N_{p,q}\to \mathbb R$
give all and only the solutions $\ell$ of~\eqref{eq:Euler-Lagrange}
connecting $p$ to $\gamma$ and having fixed energy $\kappa\leq 0$
(recall Remark~\ref{kleq0})
through the relation $\ell=F^{t^\kappa_\pm(z)}(z)$.
We are going to show that $t^\kappa_\pm$ satisfy the Palais-Smale condition
provided that $\mathcal J_{p,\gamma(t)}$
(recall Proposition~\ref{prop:restriction-Npq})
is {\em pseudocoercive}, for all $t\in\mathbb R$.
Pseudocoercivity is a compactness assumption introduced in~\cite{giannoni1999}
and recently revived in~\cite{caponio2023-calcvar}.
Let us recall it:
\begin{definition}
	Let $t, c\in \mathbb R$;
	the manifold  $\mathcal{N}_{p,\gamma(t)}$ is said to be 
	{\em $c$-precompact} if 
	every sequence $(z_n)_n \subset \mathcal J_{p,\gamma(t)}^c\coloneqq\{z\in \mathcal N_{p,\gamma(t)}:\mathcal J_{p,\gamma(t)}(z)\leq c\}$
	has a uniformly convergent subsequence.                            
	We say that $\mathcal J_{p,\gamma(t)}$ is {\em pseudocoercive} if $\mathcal N_{p,\gamma(t)}$ is $c$-precompact for all $c\in\mathbb R$.
\end{definition}
\begin{remark}
	A sufficient condition ensuring that $\mathcal J_{p,r}$ is pseudocoercive,
	for all $p,r\in M$, is based on the existence of a $C^1$ function
	$\varphi\colon M\to \mathbb R$ such that $\d\varphi(K)>0$,
	see~\cite[Proposition 8.1]{caponio2023-calcvar}.
	It is then natural to look at this result in the framework of causality properties
	of a Finsler spacetime as  global hyperbolicity. 
	We analyze this question in Appendix~\ref{equivalence}.
\end{remark}

\begin{remark}\label{cbounded}
	We point out that if $\mathcal{J}_{p,r}$ is pseudocoercive then,
	for each $c\in \mathbb{R}$,
	\[
		\sup_{z\in \mathcal J_{p,r}^c} |Q(\dot z)|<+\infty,
	\]
	see \cite[Theorem 7.6]{caponio2023-calcvar}.
\end{remark}
Henceforth, our attention turns to the functional $t_+^\kappa$,
recognizing that all the subsequent considerations
can be replicated comparably for $t_-^\kappa$.
\begin{lemma}                                                                   
	\label{Qbounded}
	Let $L$ be a Lorentz-Finsler metric, $p\in M$ and $\gamma=\gamma(t)$
	be a flow line of $K$ such that $p\not \in \gamma(\mathbb{R})$.
	Assume that $\mathcal J_{p,\gamma(t)}$ is pseudocoercive for all
	$t\in\mathbb{R}$.
	Let $(z_n)\subset \mathcal N_{p,q}$ such that $t^\kappa_+(z_n)$ is bounded,
	then
	\begin{equation}
		\label{supQ}
		\sup_{m} |Q(\dot z_n)|<+\infty.
	\end{equation}
\end{lemma}    
\begin{proof}
	Assume by contradiction that  $\sup_{m} |Q(\dot z_n)|=+\infty$.
	Since $q\in \gamma(\mathbb R)$ and 
	$\mathcal{J}_{p,\gamma(t)}$
	is pseudocoercive for all $t\in\mathbb{R}$,
	from Remark~\ref{cbounded}
	necessarily $\mathcal J_{p,q}(z_n)\to +\infty$.
	Since $t^\kappa_+(z_n)$ is bounded and   
	\begin{align}
		t_{+}^{\kappa}(z_n) &= Q(\dot z_n) + \sqrt{Q^2(\dot z_n) + 2\big( \mathcal{E}(z_n)- \kappa \big)}
		\nonumber \\
								  &=Q(\dot z_n) + \sqrt{Q^2(\dot z_n) + 2\big( \mathcal{L}(z_n)- \kappa \big)},\label{tzm}
	\end{align}
	we get that, up to pass to a subsequence, 
	\begin{equation}
		\label{Qdivneg} Q(\dot z_n)\to -\infty.
	\end{equation}
	Let $C\geq 0$ such that $|t^\kappa_+(z_n)|\leq C$, for all $m\in \mathbb N$.
	From \eqref{tzm} we then get
	\begin{equation}
		\label{2L}
		2\mathcal{L}(z_n) \leq C^2 -2 C\, Q(\dot z_n)+2\kappa.
	\end{equation}
	Take $\bar t>C$ and consider $z_n^{\bar t}\coloneqq F^{\bar t} (z_n)$.
	Recalling \eqref{eq:def-hk} and  \eqref{eq:calH-splitting},
	we then get from \eqref{2L}:
	\[
		\mathcal L(z_n^{\bar t})=
		\mathcal{L}(z_n) +\bar t Q(\dot z_n) - \frac{1}{2}\bar{t}\,^2\leq \frac{C^2} 2 +(\bar t -C)Q(\dot z_n)+\kappa\to -\infty.
	\]
	By Remark~\ref{cbounded}, we deduce that $\sup _m |Q(\dot z^{\bar t}_m)|<+\infty$.
	Then from \eqref{eq:dvLzt2},
	\[
		\sup _m |Q(\dot z_n)|<+\infty,
	\] in contradiction with \eqref{Qdivneg}.
\end{proof}
Let $z\in \mathcal N_{p,q}$ and $\zeta\in T_z\Omega_{p,q}(M)$;
we recall that the $H^1$-norm of $\zeta$ is given by
$\int_0^1g_z(\zeta', \zeta')\d s$
where $\zeta'$ denotes the covariant derivative of $\zeta$
along $z$ defined by the Levi-Civita connection of the auxiliary Riemannian metric $g$. 
\begin{lemma}
	\label{bounded}
	Let $(z_n)_n\subset \mathcal N_{p,q}$ be a bounded sequence
	$($w.r.t. the topology induced on $\mathcal N_{p,q}$
	by the topology of $\Omega_{p,q}(M)$$)$
	such that their images $z_n([0,1])$ are contained in a compact subset of $M$ and,
	for each $n\in \mathbb N$,
	let $\zeta_n\in T_{z_n}\Omega_{p,q}(M)$.
	If 
	\[
		\sup_{n} \int_0^1g_{z_n}(\zeta_n', \zeta_n')\d s<+\infty,
	\]                                                    
	then there exist bounded sequences
	$\xi_n\in T_{z_n}\mathcal N_{p,q}$ and
	$\mu_n\in H^1_0([0,1],\mathbb{R})$
	such that $\zeta_n=\xi_n + \mu_n K(z_n)$.
\end{lemma}	
\begin{proof}
	Since 
	the images of the curves $z_n$ are contained in a compact subset $W$ of $M$, 
	we can assume that the field $K$
	is bounded and the covariant derivatives
	of the fields $K(z_n)$ along $z_n$
	are uniformly bounded in the $L^2$-norm.
	Thus, it suffices to show that there exists a  bounded
	sequence $(\mu_n)_{n} \subset H^1_{0}([0,1],\mathbb{R})$
	such that 
	\[
		\xi_n \coloneqq \zeta_n - \mu_n K(z_n) \in T_{z_n}\mathcal{N}_{p,q},
		\quad \forall n \in \mathbb{N}.
	\]
	By~\eqref{eq:TzNpq}, we need to prove 
	that, for each $n\in \mathbb{N}$,
	there exists $c_n \in \mathbb{R}$ such that 
	\[
		\partial_xQ(\dot{z}_n)[\xi_n] + Q(\dot{\xi}_{n}) = c_n, \text{ a.e.},
	\]
	so we need to solve, with respect to $c_n\in \mathbb{R}$
	and $\mu_n \in H^1_0([0,1],\mathbb{R})$,
	the following ODE:
	\begin{equation}
		\label{eq:bounded-proof1}
		\partial_xQ(\dot{z}_{n})[\zeta_n]
		+ Q(\dot{\zeta}_{n})
		- \mu_n \big(\partial_xQ(\dot{z}_n)[K(z_n)] + Q(\dot{K}_n)\big)
		+ \mu_n' = c_n,
	\end{equation}
	where,
	$\dot{K}_n$ denotes
	$\frac{\partial K^i}{\partial x^j}(z_n(s))\dot z_n^j(s)\frac{\partial}{\partial x^i}|_{z_n(s)}$.
	Let us re-write \eqref{eq:bounded-proof1} as
	\begin{equation}
		\label{eq:Npqmanifold-ODE}
		\mu_n'(s) - a_n(s)\mu_n(s) = b_n(s),
	\end{equation}
	where
	\begin{align*}
		a_n(s) &= \partial_{x}Q(\dot z_n)[K(z_n)] +Q(\dot K_n)\\
		\intertext{and}
		b_n(s) &= c_n-h_n(s),
		\quad\quad
		h_n(s)\coloneqq \partial_xQ(\dot{z}_{n})[\zeta_{n}] 
		+Q(\dot{\zeta}_{n})
	\end{align*}
	Setting
	$
	A_n(s)=\int_{0}^{s}a_n(\tau)\d \tau,
	$
	and
	\[
		c_n =  \left(\int_{0}^{1}
			e^{-A_n(s)}\d s
		\right)^{-1}
		\left(\int_{0}^{1}
			e^{A_n(s)}h_n(s)\d s
		\right),
	\]
	a solution of \eqref{eq:Npqmanifold-ODE}
	which satisfies the boundary conditions 
	$\mu_n(0) = \mu_n(1) = 0$ is given by
	\begin{equation*}
		\mu_n(s)=e^{A_n(s)}\int_{0}^{s}b_n(\tau )e^{-A_n(\tau)}\d \tau.
	\end{equation*}
	We notice that the sequence $A_n(s)\colon[0,1]\to \mathbb R$
	is uniformly bounded in $L^\infty$ since
	\[
		|A_n(s)|\leq \int_0^1|a_n|\d s
		\leq C_1\int_0^1\sqrt{g(\dot z_n,\dot z_n)}\d s,
	\]
	where $C_1$ is a positive constant depending on the maxima
	of the absolute values of the components of $Q$ and $K$
	and their derivatives,
	in each  coordinate system used to cover $W$,
	and on a constant that bounds from above the Euclidean norm
	with the norm associated with $g$ in each of the same of coordinate system.
	This implies that the sequence of functions $e^{\pm A_n(s)}$
	is also uniformly bounded in $L^\infty$ and then
	$\left(\int_{0}^{1} e^{-A_n(s)}\d s	\right)^{-1}$ is bounded as well.
	Analogously, 
	\[
		|h_n(s)|\leq C_2\sqrt{g(\dot z_n,\dot z_n)},
	\]
	where now $C_2\geq 0$ is independent of $K$
	but depend on an upper bound for the $L^\infty$-norms of the fields $\zeta_n$.
	Hence $c_n$  is bounded  and $b_n$ satisfies then  
	\[
		|b_n(s)|\leq C_3 +C_2\sqrt{g(\dot z_n,\dot z_n)},
	\]
	for some non-negative constant $C_3$. Since 
	\[
		\mu'_n(s)=a_n(s)e^{A_n(s)}\int_{0}^{s}b_n(\tau )e^{-A_n(\tau)}\d \tau
		+ b_n(s)
	\]
	we get 
	\[
		|\mu'_n(s)|\leq C_4|a_n(s)|+|b_n(s)|,
	\]
	for some non-negative constant $C_4$,
	depending also on an upper bound of the sequence
	$\int_0^1\sqrt{g(\dot z_n,\dot z_n)}\d s$.
	Hence, $\mu_n$ is bounded  in $H^1_0$-norm.			
\end{proof}	

\begin{lemma}\label{dt0suW}
	Let $z\in \mathcal N_{p,q}$ and $\eta\in \mathcal W_z$,
	then $\d t^\kappa_+(z)[\eta]=0$.
\end{lemma}	
\begin{proof}
	From \eqref{eq:diff-H-in-t+}, since $\mathcal L=\mathcal E$, we get
	\[
		\partial_z\mathcal{H}_{p,q}(z,t_+^{\kappa}(z))
		+ \partial_t\mathcal{H}_{p,q}(z,t_+^{\kappa}(z))\, \d t_+^{\kappa}(z)=0.
	\]
	As showed in the proof of Proposition~\ref{prop:restriction-Hpq},
	$\partial_z\mathcal{H}_{p,q}(z,t_+^{\kappa}(z))[\eta]=0$,
	and since, from Corollary~\ref{cor:partial_tH-notnull},
	$\partial_t\mathcal{H}_{p,q}(z,t_+^{\kappa}(z))\neq 0$,
	we get the thesis.
\end{proof}

We are now ready to prove the Palais-Smale condition for $t^\kappa_+$.
We recall that a $C^1$ functional $f\colon \mathcal M\to \mathbb R$,
defined on a manifold $\mathcal M$,
satisfies the Palais-Smale condition if every sequence $z_n\subset \mathcal M$
such that $f(z_n)$ is bounded and $\d f(z_n)\to 0$,
admits a converging subsequence.
\begin{proposition}
	Under the assumptions in Lemma~\ref{Qbounded},
	$t^\kappa_+\colon\mathcal N_{p,q}\to \mathbb R$
	satisfies the Palais-Smale condition. 
\end{proposition}	
\begin{proof}
	Let $(z_n)_n\subset \mathcal N_{p,q}$  and $C\geq 0$ such that
	$|t^\kappa_+(z_n)|\leq C$ and $\d t^\kappa_+(z_n)\to 0$.
	From Lemma~\ref{Qbounded}, we have that \eqref{supQ} holds.
	Hence, from \eqref{tzm} we deduce that $\mathcal L(z_n)$
	is bounded from above. 
	By the pseudocoercivity assumption, there exists then a subsequence,
	still denoted by $z_n$,
	which uniformly converge to a continuous curve $z\colon[0,1]\to M$
	connecting $p$ to $q$.
	Thus, the curves $z_n$                                 
	are contained in a compact subset $W$ of $M$. 
	Hence, from Remark~\ref{classicalFinsler} there exists
	a positive constant $\alpha$, depending on $W$,
	such that
	$L_c(x,v)\geq \alpha  g(v,v)$,
	for all $x\in W$ and $v\in T_x M$.
	Let  $\mathcal L_c$ denote  the action functional of $L_c$ and 
	\[
		\mathcal S(z)\coloneqq \sqrt{Q^2(\dot z)
		+ 2\big( \mathcal{L}(z)- \kappa \big)}
		=\sqrt{2\big( \mathcal{L}_c(z)- \kappa \big)-Q^2(\dot z)}.
	\]
	Since  $Q(\dot{z}_n)$ is bounded,
	$\mathcal{L}(z_n)$ is  bounded from above and
	\begin{equation}
		\label{H1}
		\alpha \int_0^1g(\dot z_n,\dot z_n)\leq \mathcal L_c(z_n)
		=\mathcal L(z_n)+Q^2(\dot z_n),
	\end{equation}
	we deduce that  $\mathcal S(z_n)$ and
	$\int_0^1g(\dot z_n,\dot z_n)\d s$ are bounded as well. 
	Moreover, for $z\in \mathcal N_{p,q}$,
	let us see $Q(\dot z)$ as a functional $\mathcal{Q}$
	on $\mathcal{N}_{p,q}$
	(recall that $Q(\dot z)$ is constant a.e. on $[0,1]$). 

	Let  then  $\zeta_n\in T_{z_n}\Omega^{1,2}_{p,q}(M)$ be a bounded sequence;
	from Lemma~\ref{bounded} 
	there exist two  bounded sequences $\xi_n \in T_{z_n}\mathcal{N}_{p,q}$ 
	and $\mu_n \in H^1_0([0,1],\mathbb{R})$
	such that $\zeta_n = \xi_n + \mu_n K_{z_n}$.
	As $z_n$ is a Palais-Smale sequence, from Lemma~\ref{dt0suW}
	we  obtain 
	\[
		\d t^\kappa_+(z_n)[\zeta_n]
		=\d t^\kappa_+(z_n)[\xi_n]+\d t^\kappa_+(z_n)[\mu_nK_{z_n}]
		= \d t^\kappa_+(z_n)[\xi_n]\to 0.
	\] 	
	We  now apply a localization argument
	as in~\cite{Abbondandolo2009}
	(see also the proof of \cite[Theorem 5.6]{caponio2023-calcvar}).
	Thus, we can assume that
	$L$ is defined on $[0,1]\times U \times \mathbb R^{m+1}$,
	with $U$ an open neighbourhood of $0$ in $\mathbb R^{m+1}$.
	Analogously, we associate to   $L_c$ and $Q$, 
	a time-dependent fiberwise strongly convex Lagrangian $L_{cs}$ in $U$
	and a $C^1$ family of linear forms $Q_s$.
	Moreover, we can
	identify $(z_n)_n$ with a sequence in the Sobolev space
	$H^1([0,1],U)$.
	By \eqref{H1}, taking into account that the
	curves $z_n$ have fixed end-points,
	we get that $(z_n)_n$ is bounded in
	$H^1([0,1],U)$ and so it admits a subsequence,
	still denoted by $(z_n)_n$,
	which weakly and uniformly converges to a curve
	$z\in H^1([0,1],\mathbb R^{m+1})$
	which also satisfies the same
	fixed end-points boundary conditions. 
	The differential at $z_n$ of the localized functional obtained, that we still denote with $t_+^\kappa$,  is given by
	\[
		\d t^\kappa_+(z_n)
		=\d \mathcal{Q}_s(z_n)+\frac{\d \mathcal L_{cs}(z_n)
		-\mathcal{Q}_s(z_n)\d \mathcal{Q}_s(z_n)}{\mathcal S_s(z_n)},
	\]
	where the index $s$ is used to denote the localized functionals.
	Since $\mathcal S_s(z_n)$ is bounded we get
	\[
		0 \leftarrow \mathcal S_s(z_n)\d t^\kappa_+(z_n)
		=\big (\mathcal S_s(z_n)-\mathcal{Q}_s(z_n)\big)\d \mathcal{Q}_s(z_n)
		+\d \mathcal L_{cs}(z_n).
	\]
	In particular, since $z_n-z$ is bounded in $H^1_0$, we obtain
	\begin{equation}
		\label{to0}
		\big (\mathcal S_s(z_n)-\mathcal{Q}_s(z_n)\big) \d \mathcal{Q}_s(z_n)[z_n-z]
		+\d \mathcal L_{cs}(z_n)[z_n-z]\to 0.
	\end{equation}
	Since $z_n\to z$ uniformly and weakly, we deduce that $\d \mathcal{Q}_s(z_n)[z_n-z]\to 0$.
	As $\mathcal S_s(z_n)-\mathcal{Q}_s(z_n)$ is bounded then 
	$\big (\mathcal S_s(z_n)-\mathcal{Q}_s(z_n)\big)\d \mathcal{Q}_s(z_n)[z_n-z]\to 0$ as well. 
	From \eqref{to0}, we then get  $\d \mathcal L_{cs}(z_n)[z_n-z]\to 0$.
	We can then conclude that $z_n\to z$ in $H^1$-norm thanks to the convexity
	of $L_{cs}$ as in the proof of \cite[Theorem 5.6]{caponio2023-calcvar}.
	There exists then a subsequence $z_{n_k}$ such that
	$\dot z_{n_k}\to \dot z$, a.e. on $[0,1]$. 
	As $Q(\dot z_{n_k})=c_k$ a.e., for some $c_k\in \mathbb R$,
	we get that also $Q(\dot z)$ is constant a.e.,
	i.e. $z\in \mathcal N_{p,q}$.
\end{proof}

\begin{lemma}
	\label{lem:t_+boundedfrombelow}
	Under the assumptions of Lemma~\ref{Qbounded}, the functional
	$t^{\kappa}_+\colon \mathcal{N}_{p,q} \to \mathbb{R}$
	is bounded from below.
\end{lemma}
\begin{proof}
	By contradiction,  let us assume the existence of a 
	sequence $(z_n)_n \subset \mathcal{N}_{p,q}$ such that
	$\lim_{n \to \infty}t^{\kappa}_+(z_n)= -\infty$.
	From \eqref{tzm}, this implies  that
	\[
		\lim_{n \to \infty} Q(\dot{z}_n) = -\infty,
	\]
	hence from Remark~\ref{cbounded}, up to pass to a subsequence,  
	$\mathcal L(z_n)=\mathcal J_{p,q}(z_n) \to +\infty$.
	Therefore, from  \eqref{tzm}, $t^{\kappa}_+(z_n)\geq 0$, for $n$ big enough.
\end{proof}
We are now ready to present an existence and multiplicity results for solutions of the Euler-Lagrange equations  \eqref{eq:Euler-Lagrange}. Previous existence results, in the case $\kappa=0$, based on causality techniques
	were obtained in \cite[Proposition 6.2 and Proposition B.2]{caponio2018} for Finsler spacetimes that admit a global splitting $S\times \mathbb R$ endowed with a Lorentz-Finsler  metric of the type \eqref{eq:LF-local} and  in \cite[Theorem 2.49]{Minguz19} in  the more general setting of a manifold with  a proper cone structure.
\begin{theorem}
	\label{existmult}
	Let $M$ be a smooth, connected finite dimensional manifold,
	$L\colon TM \to \mathbb{R}$
	be a Lorentz-Finsler metric on $M$ satisfying Assumption~\ref{ass:L},
	$p\in M$ and $\gamma\colon\mathbb R\to M$ be a flow line of $K$
	such that $p\not\in \gamma(\mathbb R)$.
	Let us assume that $\mathcal J_{p,\gamma(t)}$
	is pseudocoercive for all $t\in\mathbb{R}$.	
	Let $\kappa \le 0$.  Then, 
	\begin{itemize}
		\item[(a)] There exists a curve $z\colon [0,1] \to M$
			that is a solution of Euler-Lagrange equations~\eqref{eq:Euler-Lagrange}
			with energy $\kappa$, 
			joining $p $ and $\gamma(\mathbb{R})$ and minimizes $t^\kappa_{+}$;
		\item[(b)] If $M$ is a non-contractible manifold, then there exists a sequence of  curves
			$z_n\colon [0,1] \to M$
			that are solutions of Euler-Lagrange equations~\eqref{eq:Euler-Lagrange}
			with energy $\kappa$ joining $p$ and $\gamma(\mathbb{R})$ and such that 
			$\lim_{n \to \infty}t^\kappa_{+}(w_n) = + \infty$.
	\end{itemize}
\end{theorem}
\begin{proof}
	Since $t^\kappa_{+}$ is a bounded from below, $C^1$ functional defined on a
	$C^1$ manifold and it satisfies the Palais-Smale condition,
	both part (a) and (b) follows from \cite[Theorem (3.6)]{corvellec1993},
	Theorem~\ref{theorem:almostFermat} and Remark~\ref{2hom},
	taking into account, for part (b),
	that if $M$ is non-contractible then the
	Lusternik-Schnirelmann category of $\mathcal{N}_{p,q}$ is $+\infty$
	as follows from~\cite[Proposition 6.4]{caponio2023-calcvar} and
	\cite[Proposition 3.2]{fadell1991}.
\end{proof}
\begin{remark}
	Assumption \eqref{QKmeno1} could  be considered quite restrictive;
	however,  for solutions with energy $\kappa =0$,
	that is not the case for the following reasons:
	\begin{itemize}
		\item[(1)] Since $L$ is a Lorentz-Finsler metric,
			solutions $z\colon[0,1]\to M$ of the Euler-Lagrange
			equations~\eqref{eq:Euler-Lagrange} with $\kappa=0$
			satisfy $L\big(z(s),\dot z(s)\big)=0$ for all $s\in [0,1]$
			and therefore they are  lightlike geodesics
			(see, e.g., \cite{Perlick2006365, Javaloyes2021, Minguzzi2017}).
		\item[(2)] 
				According to \cite[Proposition 4.4]{Javaloyes2014}
				(see also \cite[Proposition 3.4]{Javaloyes2021} and
				\cite[Proposition 12]{Minguzzi2017})
				for any smooth function $\varphi\colon M\to (0,+\infty)$
				and for any lightlike geodesic $z\colon[0,1]\to M$ of $L$,
				there exists a reparametrization of $z$
				(on some interval $[0,a_z]$)
				which is a lightlike geodesic of the Lorentz-Finsler
				metric $\varphi L$.
		\item[(3)] Let $\tilde L$ be a Lorentz-Finsler metric on $M$ which satisfies Assumption~\ref{ass:L} with \eqref{QKmeno1} replaced by $\tilde Q(K)<0$ (where $\tilde Q$ is the Noether charge of $\tilde L$).
			Hence, $L\coloneqq -\tilde L/\tilde Q(K)$ satisfies \eqref{QKmeno1}. 
		\item[(4)] The infinitesimal symmetry $K$ of $\tilde L$ remains an infinitesimal symmetry for $L$. This is a consequence of the fact that the flow $\psi$ of $K$  preserves $\tilde Q(K)$ (see the proof of \cite[Proposition 2.5-(iv)]{caponio2023-calcvar}), and then 
			\begin{multline*}
				\dfrac{\partial \big (L\circ \psi^c\big)}{\partial t}(t,x,v)=
				K^c(L)\big(\psi^c(t,x,v)\big)=K^c\big(-\tilde L/\tilde Q(K)\big)\big(\psi^c(t,x,v)\big)\\=
				-\dfrac{\partial \big (\tilde L\circ \psi^c/(\tilde Q(K)\circ \psi)\big)}{\partial t}(t,x,v)\\=\Big(\big(\tilde Q(K)\circ \psi\big)^{-2}\dfrac{\partial  \big(\tilde Q(K)\circ \psi\big)}{\partial t}\tilde L\circ \psi^c\\ -\big(\tilde Q(K)\circ \psi\big)^{-1}\dfrac{\partial \big (\tilde L\circ \psi^c\big)}{\partial t}\Big)(t,x,v)=0,
			\end{multline*}
	(recall the beginning of Section~\ref{sec:preliminary}).
	\end{itemize}
	Summing up,
	Theorem~\ref{existmult} also holds
	(replacing  $[0,1]$ with  unknown interval of parametrizations $[0, a_z]$)
	for a  Lorentz-Finsler metric $\tilde L$ on $M$ which satisfies
	Assumption~\ref{ass:L} with \eqref{QKmeno1} replaced by $\tilde Q(K)<0$.
\end{remark}

\begin{appendices}

\section{Affine Noether charge}
\label{sec:affineNoether}
In this section, we briefly show that Theorem~\ref{theorem:almostFermat}
holds even if the Noether charge is an affine function with respect to $v$.
Specifically,
there exists a $C^1$ one-form $Q$ on $M$
and a $C^1$ function $d\colon M \to \mathbb{R}$
such that~\eqref{eq:noether} is replaced by 
\begin{equation}
	\label{eq:def-N-affine}
	N(x,v) \coloneqq \partial_vL(x,v)[K] = Q(v) + d(x);
\end{equation}
and $d$ is invariant under the one-parameter group of
$C^3$ diffeomorphisms generated by $K$.
In such a case, the stationary type local structure is given by
\begin{equation}
	\label{eq:local-structure-affine}
	L(x,v) = L\circ \phi_* \big((y,t),(\nu,\tau)\big)
	= L_0(y,\nu) + \big(\omega_y(\nu) + d(y)\big)\tau - \frac{1}{2}\tau^2,
\end{equation}
so we have
\begin{equation}
	\label{eq:energy-local-structure-affine}
	E(x,v) = E\circ \phi_* \big((y,t),(\nu,\tau)\big)
	= E_0(y,\nu) + \omega_y(\nu) \tau - \frac{1}{2}\tau^2.
\end{equation}
Moreover, the set $\mathcal{N}_{p,r}$
is given by
\[
	\mathcal{N}_{p,r} \coloneqq
	\{z \in \Omega_{p,r}(M): \exists c \in \mathbb{R} \text{ such that }
	N(z,\dot{z}) = c, \text{ a.e. on } [0,1]\}\subset \Omega_{p,r}(M),
\]
and Proposition~\ref{prop:restriction-Npq} still holds.
Moreover, defining $F^t\colon \Omega_{p,q}(M) \to \Omega_{p,\gamma(t)}(M)$
as~\eqref{eq:def-Ft}
and $\mathcal{H}_{p,q}\colon \Omega_{p,q}(M) \times \mathbb{R}\to \mathbb{R}$
as in~\eqref{eq:def-hk},
it is possible to prove both 
Proposition~\ref{prop:restriction-Hpq}
and Corollary~\ref{cor:partial_zH}.
The main difference with the linear case is that 
Proposition~\ref{prop:LandEofzt} doesn't hold
and it is replaced by the following result,
whose proof is a based on a computation in local charts
which employs~\eqref{eq:local-structure-affine}.
\begin{proposition}
	\label{prop:LandEofzt-affine}
	For every $(x,v) \in TM$ and for every $t \in \mathbb{R}$,
	the following two equations holds:
	\[
		L\big(x,v + tK(x)\big) 
		= L\big(x,v\big) + t N(x,v) - \frac{1}{2}t^2,
	\]
	and
	\[
		E\big(x,v + tK(x)\big) 
		= E\big(x,v\big) + t Q(v) - \frac{1}{2}t^2.
	\]
	As a consequence, for every $(x,v)\in TM$ we have
	\begin{equation}
		\label{eq:diff-LandE-affine}
		L\big(x,v + tK(x)\big) - E\big(x,v + tK(x)\big) 
		=
		L\big(x,v\big) - E\big(x,v\big)
		+ t d(x).
	\end{equation}
\end{proposition}
Since on any curve $z \in \mathcal{N}_{p,q}$
the quantities $Q(\dot{z})$ and $d(z)$ are not necessarily constant,
let us introduce the functionals $\mathcal{Q}\colon \mathcal{N}_{p,q} \to \mathbb{R}$
and $\mathcal{D}\colon \mathcal{N}_{p,q} \to \mathbb{R}$
as follows:
\[
	\mathcal{Q}(z) \coloneqq \int_{0}^{1}Q(\dot{z})\d s,
	\quad\text{and}\quad
	\mathcal{D}(z) \coloneqq \int_0^1 d(z)\d s.
\]
Using this notation, for every $z\in\mathcal{N}_{p,q}$
the two quantities $t_{\pm}^{\kappa}(z)$ that satisfy~\eqref{eq:prop_t+-}
are given by
\[
	t_{\pm}^{\kappa}(z) = \mathcal{Q}(z) \pm
	\sqrt{
		\mathcal{Q}^2(z)
	+ 2\big( \mathcal{E}(z)- \kappa \big) },
\]
and they are still well defined if $\kappa$
satisfies~\eqref{eq:condition-on-kappa}.
Moreover, since the function $d$ doesn't appear in the 
expression of $E(x,v)$ in local coordinates 
(see~\eqref{eq:energy-local-structure-affine}),
Corollary~\ref{cor:partial_tH-notnull} still holds.
Because of the difference between~\eqref{eq:diff-LandE-affine}
and~\eqref{eq:diff-LandE},
we have that in the affine case
the equation analogous to~\eqref{eq:H-in-t+} is
\begin{multline*}
	\mathcal{H}_{p,q}(z,t_+^{\kappa}(z)) 
	= \mathcal{L}(z) -\mathcal{E}(z) + \kappa
	+ t_+^{\kappa}(z)\big(N(z,\dot{z}) - \mathcal{Q}(z)\big)\\
	= \mathcal{L}(z) -\mathcal{E}(z) + \kappa
	+ t_+^{\kappa}(z)\mathcal{D}(z).
\end{multline*}
As a consequence,
using a similar proof of the one of Theorem~\ref{theorem:almostFermat},
we obtain the following result.
\begin{theorem}
	\label{theorem:almostFermat-affine}
	Let $L\colon TM \to \mathbb R$ satisfy
	assumptions~\ref{ass:L}, \ref{ass:L1}, and \ref{ass:bounds},
	with~\eqref{eq:def-N-affine} instead of~\eqref{eq:noether},
	and let $\kappa \in \mathbb{R}$ satisfy~\eqref{eq:condition-on-kappa}.
	A curve $\ell$ is a solution of the Euler-Lagrange equations \eqref{eq:Euler-Lagrange}
	joining $p$ and $\gamma$ with energy $\kappa$
	if and only if
	there exists $z \in \mathcal{N}_{p,q}$ such that
	$\ell = F^{t_+^{\kappa}(z)}(z)$ or $\ell = F^{t_-^{\kappa}(z)}(z)$,
	and the following equality holds:
	\[
		\d t_+^{\kappa}(z)
		= \frac{\d \mathcal{E}(z) - \d \mathcal{L}(z) - t_+^{\kappa}(z)\d\mathcal{D}(z)}
		{\sqrt{\mathcal{Q}^2(\dot{z})+ 2\big( \mathcal{E}(z)- \kappa \big)}},
	\]
	or
	\[
		\d t_-^{\kappa}(z)
		= \frac{\d \mathcal{L}(z) - \d \mathcal{E}(z) + t_-^{\kappa}(z)\d\mathcal{D}(z)}
		{\sqrt{\mathcal{Q}^2(\dot{z})+ 2\big( \mathcal{E}(z)- \kappa \big)}}.
	\]
\end{theorem}
\begin{corollary}
	\label{cor:fermat-affine-d-const}
	Let $L\colon TM \to \mathbb R$ satisfy
	assumptions~\ref{ass:L}, \ref{ass:L1}, and \ref{ass:bounds},
	with~\eqref{eq:def-N-affine} instead of~\eqref{eq:noether},
	and let $\kappa \in \mathbb{R}$ satisfy~\eqref{eq:condition-on-kappa}.
	Moreover, assume that $d\colon M \to \mathbb{R}$ is a constant function.
	Then, a curve $\ell$ is a solution of the
	Euler-Lagrange equations~\eqref{eq:Euler-Lagrange}
	joining $p$ and $\gamma$ with energy $\kappa$
	if and only if
	there exists $z \in \mathcal{N}_{p,q}$ such that
	$\ell = F^{t_+^{\kappa}(z)}(z)$ or $\ell = F^{t_-^{\kappa}(z)}(z)$,
	and~\eqref{eq:t_+-critical?} or~\eqref{eq:t_--critical?} holds.

\end{corollary}

\section{Pseudocoercivity and global hyperbolicity}
\label{equivalence}
In this appendix we show that pseudocoercivity and global hyperbolicity of a Finsler spacetime $(M,L)$
as defined in Subsection~\ref{sec:LF},
are connected notions.
We refer to \cite{Minguz15, Minguz19} for the needed notions of  causality,
and in particular of global hyperbolicity and of a Cauchy hypersurface,
in Finsler spacetimes and in the more general framework
of proper cone structures
(see \cite[Definition~2.4]{Minguz19}).
We notice indeed that $M$ is endowed with a continuous  cone structure 
$\mathcal C\coloneqq \{(x,v)\in TM: L(x,v)\leq 0, Q(v)<0\}$. 
In fact, from the local expression of $L$  \eqref{eq:LF-local},
we deduce that $(\nu,\tau)\in \mathcal C_{(y,t)}\coloneqq\mathcal C\cap T_{(y,t)}M$
if and only if  
\begin{equation}
	\label{causal}
	\tau\geq \omega(\nu)+\sqrt{\omega^2(\nu) +2F^2(y,\nu)},\
\end{equation}
and since  $F^2(y,\cdot)$ is strongly convex,  we deduce that 
$\mathcal{C}_{(y,t)}\cup\{0\}$ is a  closed, convex, sharp cone with non-empty interior.

Our first aim  would be to extend \cite[Theorem 5.1]{CaFlSa08},
which states that if a stationary Lorentzian manifold is globally hyperbolic
with a complete Cauchy hypersurface then it is pseudocoercive.
We  obtain a  result in that direction,
namely Proposition~\ref{prop:fromGH-toPseudoCoerciviness},
that ensures pseudocoerciveness from the global hyperbolicity
in our setting requiring some other technical assumptions
that are trivially satisfied in the Lorentzian setting.

\begin{lemma}
	\label{ghtosplitting}
	Let $(M,L)$ be a Finsler spacetime
	$($i.e. $L_F \colon TM \to \mathbb{R}$ satisfies $(a')$,
	$(b)$ and $(c')$ in Definition~\ref{def:Lorentz-Finsler-metric}$)$
	such that Assumption~\ref{ass:L} holds.
	If $(M,L)$ is globally hyperbolic
	$($i.e. the cone structure $\mathcal C$ associated to $L$ is globally hyperbolic$)$
	then $M$ globally splits as $S\times \mathbb R$ and $L$ is given on $S\times \mathbb R$
	by an expression of the type \eqref{eq:L-localchart},
	with $L_0\coloneqq L|_{TS}$ and $\omega$ the one-form induced by $Q$ on $S$.
\end{lemma}
\begin{proof}
	From \cite[Theorem 1.3]{FatSic12},
	we have that there exists a smooth Cauchy time function
	$T\colon M\to \mathbb R$.
	Let then $S\coloneqq T^{-1}(0)$.
	Being $K_x\in \mathcal C_x$, for all $x\in M$, we have that $\d T(K)>0$
	by definition of a smooth time function,
	and then  $K$ is transversal to $S$.
	Thus, for any vector $(x,w)\in TM$ with $x\in S$,
	we can write $w=w_S+\tau_w K_x$ where $w_S\in T_x S$.
	Since
	\[
		\frac{\d}{\d s} L(x,w_S + s\tau_w K)=\tau_w\partial_vL(x, w_S+s\tau_w K)[K]=\tau_w Q(w_S)-s\tau_w^2,
	\]
	by integrating w.r.t. $s$ between $0$ and $1$, we get
	\[
		L(x, w)=L(x, w_S+\tau_w K)=L(x,w_S)+\tau_w Q(w_S)-\frac 12 \tau_w^2,
	\]
	which gives the required expression for $L$ restricted to vectors $(x,w)\in TM$ with $x\in S$.
	Let $\phi$ be  the restriction to $S\times \mathbb R$ of the flow of $K$.
	Since $T$ is a Cauchy time function,
	it is strictly increasing on the flow lines $\gamma$ of $K$
	and it satisfies $\lim_{s\to \pm \infty}T(\gamma(s))=\pm \infty$.
	Therefore, $\phi\colon S\times \mathbb R\to M$ is a diffeomorphism.
	Using that $L$ is invariant by the flow of $K^c$
	we obtain 
	\begin{equation}
		\label{global}
		L \circ \phi_*\big((x,t),(\nu,\tau)\big)
		= L_0(x,\nu) + \omega(\nu)\tau - \frac 1 2 \tau^2,
	\end{equation}
	where  $L_0=L|_{TS}$ and 
	$\omega$ is the one-form induced by $Q$ on $S$.
\end{proof}

Let us denote by $g_S$ the $C^1$  Riemannian metric
on $S$ induced by $g$.
We assume that the one-form $\omega$ has sublinear growth w.r.t. the distance $d_{S}$ induced by
$g_S$, i.e. there exist  $\alpha\in [0,1)$ and two non-negative constants
$k_0$ and $k_1$ such that
\begin{equation}
	\label{sublinear}
	\|\omega\|\leq k_0+k_1\big(d_S(x,x_0)\big)^\alpha,
\end{equation}
for some $x_0\in S$ and all $x\in S$. 
From \cite[Proposition 8.1]{caponio2023-calcvar},
we immediately  obtain the following result.
\begin{proposition}
	\label{prop:fromGH-toPseudoCoerciviness}
	Under the assumptions of Lemma~\ref{ghtosplitting},
	assume also that  $g$ is complete, \eqref{sublinear} holds, $L_0$   is non-negative and  satisfies 
	\begin{equation}\label{L0convex}	\big(\partial_v L_0(x,v_2)-\partial_0 L_0(x,v_1)\big)[v_2-v_1]
		\geq \lambda_0(x)\norm{v_2 - v_1}^2,
	\end{equation}
	for each  $x\in S$, and all $v_1, v_2\in T_xS$. If 
	$\displaystyle \inf_{x\in S}\lambda_0(x)>0$,  then 
	$\mathcal J_{p,r}$ is pseudocoercive for all $p, r\in M$.
\end{proposition}
\begin{remark}
	\label{noB3}
	We notice that the condition~\eqref{L0convex} is always  satisfied in the Lorentzian setting, 
	since $S$ can be taken to be a smooth spacelike Cauchy hypersurface;
	moreover if $S$ is complete then a possible auxiliary Riemannian metric
	$g$ on $S\times\mathbb{R}$ is  the natural product metric
	which is then also complete.
	Therefore,  $\lambda_0(x)=1$, for each $x\in S$, in the Lorentzian setting.
	We point out that in~\cite[Theorem 5.1]{CaFlSa08}
	the completeness of $S$ is a required assumption.
	The more technical assumption in
	Proposition~\ref{prop:fromGH-toPseudoCoerciviness} is \eqref{sublinear}.
	It is needed to get the boundedness of the constants $Q(\dot z)$,
	for all $z$ in a fixed sublevel $\mathcal J_{p,r}^c$,
	a property called {\em c-boundedness} in \cite{caponio2023-calcvar},
	that implies pseudocoerciveness if satisfied for each $c\in\mathbb R$
	(see~\cite[Proposition 7.2]{caponio2023-calcvar}).
	Actually, when $L$ is a $2$-positive homogeneous Lagrangian
	and $L_0\in C^1(TS)$ is the square of a Finsler metric on $S$,
	a close inspection of the proof of~\cite[Theorem 5.1]{CaFlSa08}
	makes clear that~\eqref{sublinear} can be removed,
	and an analogous proof can be repeated
	by using the action functional of $L_0$
	instead of the energy functional of the Riemannian metric on $S$.
	In fact, using the global splitting $S\times \mathbb R$ and \eqref{global},
	the arrival time functional of a lightlike curve
	$z(s)=\big (x(s), t(s)\big)$,
	(i.e., a causal curve $z\colon [0,1]\to M$ such that
	$L\big (z(s), \dot z(s)\big )=0$, a.e. on $[0,1]$)
	between $p=(x_0, 0)\in S\times\{0\}$ and a flow line of $K$,
	$\gamma(t)=(x_1, t)$, is given by
	\[
		J\colon \Omega_{x_0,x_1}(S)\to \mathbb{R},
		\qquad
		J(x)=\int_0^1
	\left(\omega(\dot x)+\sqrt{\omega^2(\dot x) +2L_0(x, \dot x\big)}\right)\d s,
	\]
	and this is a key point in the proof of~\cite[Theorem 5.1]{CaFlSa08}
	(refer to~\cite[Lemma 5.4]{CaFlSa08}).
	Moreover, the completeness of the Riemannian metric on $S$
	can be replaced by the forward or backward completeness of $\sqrt{L_0}$.
	Another fundamental point is the compactness of $S\cap J^-(q)$, for any $q\in M$,
	(see \eqref{Jp} for the definition of $J^-(q)$),
	used in the proof of \cite[Lemma 5.5]{CaFlSa08}.
	In our setting, this is an immediate consequence
	of~\cite[Theorem 2.44]{Minguz19}.
	Summing up, the following result extending~\cite[Theorem 5.1]{CaFlSa08} holds:
\end{remark}
\begin{theorem}
	\label{th:fromGH-toPseudoCoerciviness}
	Under the assumptions of Lemma~\ref{ghtosplitting},
	assume also that $L_0\in C^1(TS)$ is the square of a forward or backward complete Finsler metric  on $S$.
	Then $\mathcal J_{p,r}$ is pseudocoercive for all $p, r\in M$.
\end{theorem}

\begin{remark}
	In light of Theorem~\ref{th:fromGH-toPseudoCoerciviness},  it becomes important to give conditions ensuring that $L_0$ is the square of a Finsler metric on $S$.
	A first observation is that $L_0$ is non-negative and \eqref{L0convex} holds if, for each $x\in S$
	\begin{equation*}
		\lambda(x) -\max_{\nu\in T_x S, \|\nu\|=1} 2Q^2_x(\nu)>0,
	\end{equation*}
	where $\lambda(x)$ is defined in \eqref{eq:monotone} (see \cite[Remark 2.14]{caponio2023-calcvar}).

	We also notice that, if $\mathcal O_0:=\mathcal O\cap TS$, satisfies, relatively to $TS$, the same properties satisfied 
	by $\mathcal O$ in  Remark~\ref{weakFinsler}-(a'), then \eqref{L0convex} holds  if
	\begin{equation}
		\label{strong}
		\inf_{v\in\mathcal O_x}\left(\min_{\nu\in T_x S, \|\nu\|=1}\big(\partial_{vv}L_c(x,v)[\nu,\nu]- 2 Q^2_x(\nu)\big)\right)>0.
	\end{equation}
	Moreover, in this case,  $\sqrt{L_0}$ in \eqref{global} is a Finsler metric on $S$
	such that $L_0$ is of class $C^1$.
	Indeed, from \eqref{HessLc}  and \eqref{strong} we immediately get that
	$\partial_{vv}L(x,v)|_{T_xS\times T_x S}$ is a positive definite bilinear form,
	for every $v \in T_xS\cap \mathcal O_0$.
	Therefore,  recalling that $L_0=L|_{TS}$ and it is fiberwise positively  homogeneous, we have that $L_0(v)\geq 0$ for all $v\in  \mathcal O_0$ and then on $TS$ by density of $\mathcal O_0$ in $TS$. Arguing as in Remark~\ref{classicalFinsler}, we then conclude 
	that $\sqrt{L_0}$ is a Finsler metric.

	Actually, in this last setting, \eqref{strong} is also a necessary condition for $L_0$ being the square of a Finsler metric. In fact, let $\{e_1,\dots,e_m\}\subset T_xS$ be an orthonormal basis of $T_xS$
	with respect to the auxiliary Riemannian metric $g$.
	Using this basis,
	we can write the one-form $\omega\colon T_xS \to \mathbb{R}$
	given by $Q|_{T_xS}$ as $(\omega_1,\dots,\omega_m)$.
	Let us denote by 
	$g_0(v)_{ij}$ the vertical Hessian matrix of $L_0$ in $v\in T_xS\cap \mathcal O_x$
	with respect to this basis.
	Similarly, we denote by $g_c(v)_{ij}$ the vertical Hessian matrix of $L_c$
	restricted to $T_xS$.
	With this notation, 
	first we notice that, $g_0(v)_{ij}$ has  $m-1$ positive eigenvalues,
	since it coincides with $g_c(v)_{ij}$ on $\mathrm{ker}(\omega)$.
	By~\cite[Proposition 11.2.1]{BaChSh00}, applied to the vector  $i\sqrt{2}(w_1, \ldots, w_m)\in \mathbb C^m$,
	we have 
	\[
		\mathrm{det}(g_0(v)_{ij})=\mathrm{det}\big(g_c(v)_{ij} -2\omega_i\omega_j\big)
		=\big(1 - 2g_c(v)^{ih}\omega_h\omega_i\big)\mathrm{det}(g_c(v)_{ij}),
	\]
	where $g_c(v)^{ij}$ denotes the inverse matrix of $g_c(v)_{ij}$.
	Since $g_c(v)_{ij}$ is positive definite,
	then $g_0(v)_{ij}$ is positive definite if and only if
	$1-2g_c(v)^{ih}\omega_h\omega_i > 0$,
	namely if and only if the norm of $\omega$
	with respect to $g_c(v)$ is strictly less than $1/2$
	for every $v\in T_xS\cap \mathcal O_x$.
\end{remark}

Let us now analyze the converse situation,
i.e. we assume now that $\mathcal J_{p,r}$ is pseudocoercive for all $p, r\in M$
and we prove that global hyperbolicity holds.
We recall (see, e.g., \cite[\S 2.1]{Minguz19})
that an absolutely continuous curve $\gamma\colon [a,b] \to M$ is \emph{causal} 
if $\dot \gamma (t)\in \mathcal C_{\gamma(t)}$,
for a.e. $t\in[a,b]$.
For any $p \in M$,
we set
\[	J^+(p) \coloneqq \left\{
		\begin{aligned}
			r \in M:
			r = p &
			\text{ or there exists a causal curve } \gamma\colon [0,1] \to M\\
					&
					\text{such that $\gamma(0) = p$ and $\gamma(1) = r$
					}
				\end{aligned}
			\right\},
		\]
		and, analogously, we define
		\begin{equation}\label{Jp}
			J^-(p) \coloneqq \left\{
				\begin{aligned}
					r \in M:
					r = p &
					\text{ or there exists a causal curve } \gamma\colon [0,1] \to M\\
							&
							\text{such that $\gamma(0) = r$ and $\gamma(1) = p$
							}
						\end{aligned}
					\right\}.
				\end{equation}
				We call \emph{causal diamond} 
				a set given by $J^+(p)\cap J^-(r)$,
				for some $p,r \in M$.

				According to \cite[Corollary 2.4]{Minguz19},
				global hyperbolicity on a proper  cone structure $\mathcal C$
				is equivalent to the non-existence of absolutely continuous closed causal curves
				plus compactness of every causal diamond.
				We use this characterization to prove the next result
				that extends to Lorentz-Finsler stationary spacetimes \cite[Proposition B1]{giannoni1999}.
				\begin{theorem}
					\label{pseudocoercivity--gh}
					Let $(M,L)$ be a Finsler spacetime 
					such that Assumption~\ref{ass:L} holds.
					If $\mathcal J_{p,r}$ is pseudocoercive for all $p, r\in M$,
					then $(M,L)$ is globally hyperbolic.
				\end{theorem}

				\medskip
				Before proving the above result we need the following lemma.
				\begin{lemma}
					\label{H1}
					Any absolutely continuous causal curve $\gamma\colon [a,b]\to M$
					admits a reparametrization on $[0,1]$
					as an $H^1$ curve with $Q(\dot \gamma(s))=\mathrm{const.}$.
				\end{lemma}
				\begin{proof}
					By the local splitting and homogeneity in \eqref{causal},
					we can use the locally defined  functions $t$ to parametrize locally $\gamma$ as $\gamma(t)=(x(t),t)$,
					so  that $t\mapsto \norm{\dot x(t)}$ is locally bounded.
					As the support of $\gamma$ is compact,
					we can patch together the locally defined reparametrization to get an $H^1$
					curve defined on an interval $[0,c]$,
					and a further reparametrization gives the thesis.
				\end{proof}
				\begin{proof}[Proof of Proposition~\ref{pseudocoercivity--gh}]
					From Lemma~\ref{H1}, there is no loss of generality in considering just $H^1$
					curves parametrized on $[0,1]$ with  $Q(\dot \gamma(s))=\mathrm{const.}$.
					Assume that there exists a 	closed causal curve $\gamma\colon[0,1]\to M$.
					We take the sequence $\gamma_n$, $n\geq 1$, 
					defined by concatenating the $n$ curves
					$\gamma_j(s)\coloneqq \gamma(n(s-j/n))$ for $s\in [j/n, (j+1)/n]$, $j=0, \dots, n-1$.
					The sequence satisfies $\mathcal J(\gamma_n)\leq 0$
					but it does not admit any uniformly converging subsequence
					in contradiction with pseucoercivity of $\mathcal J_{\gamma(0), \gamma(0)}$,
					hence $(M,L)$  must be causal.
					Let us now assume by contradiction that $J^+(p)\cap J^-(r)$ is not compact.
					Then there exists a sequence  of points $(q_n)_{n \in \mathbb{N}}\subset J^+(p)\cap J^-(r)$
					that does not admit any subsequence converging to a point in $J^+(p)\cap J^-(r)$.
					We take then  a sequence of causal curves
					$(\gamma_n)_{n \in \mathbb{N}}\subset  J^+(p)\cap J^-(r)$ 
					such that $q_n\in \gamma_n([0,1])$,
					for each $n\in \mathbb N$.
					Moreover, by Lemma~\ref{H1} we can assume that 
					the sequence
					$(\gamma_n)_{n \in \mathbb{N}}$ belongs to $\mathcal J_{p,r}$.
					Since $\mathcal{J}(\gamma_n)\le 0$ for every $n \in \mathbb{N}$,
					by pseudocoercivity $(\gamma_n)_{n \in \mathbb{N}}$
					admits a uniformly converging subsequence $(\gamma_{n_k})_{k}$.
					The uniform limit is then a causal curve $\gamma\colon [0,1]\to M$ connecting $p$ to $r$, by theorem~\cite[Theorem 2.12]{Minguz19}.
					This implies  
					that  $(q_{n_k})_{k}$ must admit a converging subsequence to a point in $J^+(p)\cap J^-(r)$,
					which is a contradiction.
				\end{proof}
\end{appendices}


\end{document}